\documentclass[reqno,11pt]{amsart}
\usepackage{amsmath, amsfonts,amsthm,amssymb,amsbsy,upref,color,graphicx,amscd,hyperref,enumerate,xfrac}
\usepackage[active]{srcltx}
\usepackage[latin1]{inputenc}

\usepackage{pgf}
\usepackage{tikz}
\usetikzlibrary{patterns}
\usepackage{sansmath}
\usetikzlibrary{shadings,intersections,patterns}
\usetikzlibrary{intersections, calc, angles}
\usepackage{tkz-euclide}
\usepackage{pgfplots}
\pgfplotsset{compat=1.10}
\usepgfplotslibrary{fillbetween}
\usepackage{xcolor,xfrac}
\definecolor{racing}{rgb}{0.7,0.1,0.2}
\definecolor{french}{rgb}{0,0.2,0.7}
\usepackage{verbatim}
\usepackage[shortlabels]{enumitem}

\usepackage{hyperref} 
\usepackage[capitalise, nameinlink, noabbrev]{cleveref} 
\hypersetup{colorlinks=true, 
	linkcolor=red,
	citecolor=blue,
}

\usetikzlibrary{arrows,automata}
\usetikzlibrary{positioning}
\usetikzlibrary{calc,through}

\tikzset{
	state/.style={
		rectangle,
		rounded corners,
		draw=black, very thick,
		minimum height=2em,
		inner sep=2pt,
		text centered,
	},
}
\newcommand{\mean}[1]{\,-\hskip-1.08em\int_{#1}}

\def\ds{\displaystyle}
\def\eps{{\varepsilon}}
\def\N{\mathbb{N}}
\def\O{\Omega}

\def\R{\mathbb{R}}
\def\Q{\mathbb{Q}}
\def\A{\mathcal{A}}
\def\V{\mathcal{V}}

\def\HH{\mathcal{H}}

\newcommand{\be}{\begin{equation}}
\newcommand{\ee}{\end{equation}}

\newcommand{\cp}{\mathop{\rm cap}\nolimits}

\numberwithin{equation}{section}
\theoremstyle{plain}
\newtheorem{teo}{Theorem}[section]
\newtheorem{lemma}[teo]{Lemma}
\newtheorem{cor}[teo]{Corollary}
\newtheorem{prop}[teo]{Proposition}

\theoremstyle{remark}
\newtheorem{oss}[teo]{Remark}
\newtheorem{exam}[teo]{Example}

\theoremstyle{plain}

\newcommand{\mathbbmm}[1]{\text{\usefont{U}{bbm}{m}{n}#1}}
\newcommand{\ind}{\mathbbmm{1}}

\title{Free boundary cluster with Robin condition on the transmission interface}

\makeindex
\author[S. Guarino Lo Bianco, D. A. La Manna, B. Velichkov]{
	Serena Guarino Lo Bianco, Domenico Angelo La Manna, Bozhidar Velichkov}
\address{
	Serena Guarino Lo Bianco\newline \indent
	Universit\`a degli studi di Napoli ``Federico II''\newline \indent
	Dipartimento di Agraria\newline \indent
	Via Universit\`a 100,
	80055 Portici (NA)- ITALY.
}
\email{ serena.guarinolobianco@unina.it}

\address{
Domenico Angelo La Manna\newline \indent
Universit\`a degli studi di Napoli ``Federico II''\newline \indent
Dipartimento di Matematica "Renato Caccioppoli"\newline \indent 
Via Cintia, Monte S. Angelo - 80126 Napoli - ITALY}
\email{domenicolamanna@hotmail.it}
\address {Bozhidar Velichkov: \newline \indent
	Dipartimento di Matematica, Universit\`a di Pisa \newline \indent
	Largo Bruno Pontecorvo, 5, 56127 Pisa - ITALY}
\email{bozhidar.velichkov@unipi.it}

\begin{document}

	\begin{abstract}
		We formulate and study a variational two-phase free boundary problem with Robin condition on the interface between the two phases, and we prove existence and regularity of solutions in dimension two.
	\end{abstract}
	
\maketitle	

\section{Introduction}
Free boundary problems with two  and more phases are often used to describe models in different areas of Physics, Engineering and Life Sciences, for instance in Fluid Dynamics (Bernoulli free boundary problems), Dynamics of Populations (optimal partitions problems), Mechanics and Phase Transition (obstacle problems). The different phases are called \emph{segregated} if they occupy different space regions; segregation occurs for instance in the two-phase Bernoulli problem, the two-phase obstacle problem,  optimal partitions problems. \medskip

\noindent In all these cases the interaction between the different phases is supposed to be \emph{competitive}, in particular, the interfaces are not formed because it is convenient energetically, but due to the lack of space. For instance, if we have two disjoint one-phase solutions of the variational Bernoulli (or obstacle) problem, then the couple they form is a minimizer to the corresponding two-phase problem, and even if the two phases are very close to each other, an interface is not formed (we briefly discuss this phenomenon in \cref{sub:intro-bernoulli}). \medskip

\noindent In this paper, we consider a two-phase problem, in which the phases are still \emph{segregated}, but the interaction along the free interface is \emph{collaborative}. In this case, if two or more disjoint one-phase solutions are sufficiently close, then it is energetically convenient for them to create a free interface, that is, the formation of clusteres is incentivized.\medskip

\noindent We introduce the functional in \cref{sub:intro-robin}, while in \cref{sub:statement} we state the variational problem and the main results of the paper.

\subsection{The classical one-phase and two-phase Bernoulli free boundary problems}\label{sub:intro-bernoulli}
Let $D$ be a smooth bounded open set in $\R^d$. Let $g:\partial D\to\R$ be a given nonnegative function and $\lambda>0$ a given constant. The classical one-phase Bernoulli problem can be stated as follows. Find a domain $\Omega\subset D$ and a function $u:D\to\R$ such that $u=g$ on $\partial D$ and 
$$\Delta u=0\quad\text{in}\quad \Omega\,,\qquad	u=0\quad\text{and}\quad|\nabla u|=\lambda\quad\text{on}\quad \partial\Omega\cap D\,.$$
In the seminal paper \cite{altcaf} Alt and Caffarelli showed that the existence of such a couple $(u,\Omega)$ can be obtained by minimizing the functional 
$$J_\lambda(u)=\int_{D}|\nabla u|^2+\lambda^2|\{u>0\}\cap D|,$$
among all functions in $H^1(D)$ such that $u=g$ on $\partial D$, and then taking $\Omega:=\{u>0\}$.\medskip

In the two-phase problem, we are given two constants $\lambda_1>0$ and $\lambda_2>0$ and two nonnegative functions $g_1:\partial D\to\R$ and $g_2:\partial D\to \R$ with disjoint supports. Then, the two-phase free boundary problem is the following. Find two \emph{disjoint} sets $\Omega_1$ and $\Omega_2$ in $D$ and two functions $u_1:D\to\R$ and $u_2:D\to\R$ such that $u_1=g_1$ and $u_2=g_2$ on $\partial D$ and  
 \begin{equation}
 \begin{cases}
 \begin{array}{ll} 
 \Delta u_1=0\quad\text{in}\quad\Omega_1\qquad\text{and}\qquad	 \Delta u_2=0\quad\text{in}\quad\Omega_2\,;\\
 u_1=0\quad\text{and}\quad|\nabla u_1|=\lambda_1\quad\text{on}\quad D\cap\partial\Omega_1\setminus\partial\Omega_2\,;\\
 u_2=0\quad\text{and}\quad|\nabla u_2|=\lambda_2\quad\text{on}\quad D\cap\partial\Omega_2\setminus\partial\Omega_1\,;\\
u_1=u_2=0\quad\text{and}\quad |\nabla u_1|^2-|\nabla u_2|^2=\lambda_1^2-\lambda_2^2\quad\text{on}\quad D\cap\partial\Omega_1\cap\partial\Omega_2\,.
 \end{array}
 \end{cases}
\end{equation}	
The existence of a solution can be obtained by minimizing the functional 
$$J_{\lambda_1,\lambda_2}(u)=\int_{D}|\nabla u|^2+\lambda_1^2|\{u>0\}\cap D|+\lambda_2^2|\{u<0\}\cap D|,$$
among all functions in $H^1(D)$ such that $u=g_1-g_2$ on $\partial D$, and then taking $u_1=\max\{u,0\}$, $u_2=\max\{-u,0\}$, $\Omega_1=\{u_1>0\}$ and $\Omega_2=\{u_2>0\}$ (see \cite{alcafr}, \cite{SV} and \cite{DSV}).\medskip

In the two-phase problem, the two-phase interface $\partial\Omega_1\cap\partial\Omega_2$ is formed when the two sets $\Omega_1$ and $\Omega_2$ act as geometric obstacles to each other; if $\Omega_1$ and $\Omega_2$ are disjoint one-phase solutions, then the two-phase interface is simply not formed.  In other words, if $u_1$ and $u_2$ are minimizers of the one-phase functionals $J_{\lambda_1}$ and $J_{\lambda_2}$ such that $u_1u_2\equiv0$, then it is immediate to check that $u=u_1-u_2$ is a minimizer of the two-phase functional $J_{\lambda_1,\lambda_2}$. In fact, if $v\in H^1(D)$ is such that $v=u$ on $\partial D$, then $v_+=u_1$ and $v_-=u_2$ on $\partial D$ and so, by the optimality of $u_1$ and $u_2$, we get
$$J_{\lambda_1,\lambda_2}(v)=J_{\lambda_1}(v_+)+J_{\lambda_2}(v_-)\ge J_{\lambda_1}(u_1)+J_{\lambda_2}(u_2)=J_{\lambda_1,\lambda_2}(u).$$


\subsection{A two-phase problem with Robin condition on the free interface}\label{sub:intro-robin} 
In this paper we study a different type of two-phase problem in which the two state functions $u_1$ and $u_2$ might not vanish on the interface $\partial\Omega_1\cap\partial\Omega_2$. Precisely, given $\beta>0$, $\Lambda>0$ and a fixed set $D$, we consider the functional 
$$J_{\beta,\Lambda}(u,\Omega_1,\Omega_2)=\int_{D}|\nabla u|^2\,dx+\beta\int_{\partial\Omega_1\cap\partial\Omega_2}u^2\,d\HH^{d-1}+\Lambda|\{u>0\}\cap D|,$$
defined for couples of disjoint domains $\Omega_1,\Omega_2$ in $D$ and functions $u\in H^1(D)$ with $u=0$ on $D\setminus(\Omega_1\cup\Omega_2)$. We will then show that if $\big(\Omega_1, \Omega_2, u\big)$ locally minimize $J_{\beta,\Lambda}$ in $D$, then the functions
$$u_1=u\ind_{\Omega_1}\qquad\text{and}\qquad u_2=u\ind_{\Omega_2}\ ,$$
are solutions to the problem
\begin{equation}
\begin{cases}
\begin{array}{ll} 
\Delta u_1=0\quad\text{in}\quad\Omega_1\qquad\text{and}\qquad	 \Delta u_2=0\quad\text{in}\quad\Omega_2\,\\
u_1=0\quad\text{and}\quad|\nabla u_1|=\sqrt{\Lambda}\quad\text{on}\quad D\cap\partial\Omega_1\setminus\partial\Omega_2\,\\
u_2=0\quad\text{and}\quad|\nabla u_2|=\sqrt\Lambda\quad\text{on}\quad D\cap\partial\Omega_2\setminus\partial\Omega_1\,\\
u_1=u_2\quad\text{and}\quad |\nabla u_1|+|\nabla u_2|=\beta (u_1+u_2)\quad\text{on}\quad D\cap\partial\Omega_1\cap\partial\Omega_2\,,
\end{array}
\end{cases}
\end{equation}	
and satisfy an additional condition on $\partial\Omega_1\cap\partial\Omega_2$ involving the mean curvature of the interface (see \cite{GLMV}). \medskip

Notice that if $(u_1,\Omega_1)$ and $(u_2,\Omega_2)$ are two minimizers of the one-phase Bernoulli functional $J_{\sqrt\Lambda}$ with disjoint supports ($\Omega_1\cap\Omega_2=\emptyset$), the triple ($\Omega_1$, $\Omega_2$, $u=u_1+u_2$) might not be a minimizer of $J_{\beta,\Lambda}$, even if the Hausdorff distance between $\Omega_1$ and $\Omega_2$ is strictly positive. In fact, it might be convenient to enlarge the domains $\Omega_1$ and $\Omega_2$ in order to obtain a non-empty interface $\partial\Omega_1\cap\partial\Omega_2$ that will allow to have competitors which are not vanishing identically on the entire free boundaries $\partial\Omega_1$ and $\partial\Omega_2$. This is illustrated by the following one-dimensional example.

\begin{exam}[Formation of an interface in 1D]
Let $\eps>0$ and $\beta>0$ be fixed. We consider the interval 
$D=[-1-\eps,1+\eps]$ and the boundary data $g_1, g_2:\partial D\to\R$ given by 
$$g_1(-1-\eps)=1\,,\quad g_1(1+\eps)=0\,,\quad g_2(-1-\eps)=0\,,\quad g_2(1+\eps)=1.$$
The minimizers of the one-phase function 
$$J_1=\int_D|u'(x)|^2\,dx+|\{u>0\}\cap D|$$
with boundary conditions $g_1$ and $g_2$ are respectively the functions
$$u_1(x)=(-x-\eps)_+\qquad\text{and}\qquad u_2(x)=(x-\eps)_+.$$
If we consider the sets $\Omega_1=(-1-\eps,-\eps)$ and $\Omega_2=(\eps,1+\eps)$, then we have that 
$$J_{\beta,1}(u_1+u_2,\Omega_1,\Omega_2)=J_1(u_1)+J_1(u_2)=4.$$
On the other hand, by taking
$$\widetilde \Omega_1=[-1-\eps,0]\ ,\quad \widetilde \Omega_2=[0,1+\eps]\ , \quad u(x)=\begin{cases}
1\quad\text{if}\quad x=-1-\eps,\\
\ell\quad\text{if}\quad x=0,\\
1\quad\text{if}\quad x=1+\eps,
\end{cases}$$
and extending $u$ linearly on the intervals $[-1-\eps,0]$ and $[0,1+\eps]$
we obtain that 
$$J_{\beta,1}\big(u,\widetilde\Omega_1,\widetilde\Omega_2\big)=2\frac{(1-\ell)^2}{1+\eps}+\beta\ell^2+2+2\eps.$$
Setting the parameter $\ell$ to be the optimal one, 
$\ds \ell=\frac{2}{2+\beta+\eps\beta},$
we get that
$$J_{\beta,1}\big(u,\widetilde\Omega_1,\widetilde\Omega_2\big)=2\frac{(1+\eps)\beta^2}{(2+\beta+\eps\beta)^2}+\beta\left(\frac{2}{2+\beta+\eps\beta}\right)^2+2+2\eps.$$
When $\eps=0$, we get
$$J_{\beta,1}\big(u,\widetilde\Omega_1,\widetilde\Omega_2\big)=2\frac{\beta^2}{(2+\beta)^2}+\beta\left(\frac{2}{2+\beta}\right)^2+2=2\frac{\beta^2+2\beta}{\beta^2+4\beta+4}+2<4.$$
In conclusion, if we fix $\beta>0$ we can find $\eps_0>0$ such that 
$$J_{\beta,1}\big(u,\widetilde\Omega_1,\widetilde\Omega_2\big)<4\qquad\text{for all}\qquad \eps\in(0,\eps_0),$$
which means that for those choices of $\beta$ and $\eps$, the combination of the two one-phase solutions is not optimal. 
\end{exam}

\subsection{Setting of the problem and main theorem}\label{sub:statement}
We will define the variational problem for the functional $J_{\beta,\Lambda}$ in the class of sets of finite perimeter and Sobolev functions. Then, we will prove an existence theorem in this class and we will show that the minimizers are regular. We fix the boundary data for $\Omega_1$, $\Omega_2$ and $g$. Precisely, let 
\begin{itemize}
\item $E_1$ and $E_2$ be two smooth, bounded and disjoint sets of positive distance in $\R^d$;
\item $D:=\R^d\setminus \Big(\overline E_1\cup\overline E_2\Big)$;
\item $\O_i=E_i$ in $\R^d\setminus D$;
\item $g\in H^1(\R^d)\cap L^\infty(\R^d)$ be a non-negative function such that 
$$ g\equiv 1\quad\text{on}\quad E_1\cup E_2\,.$$
\end{itemize}	

\noindent We define the following admissible set of functions 
$$\V=\Big\{u\in H^1(\R^d) \,:\, u\ge 0\,\text{ in }\,\R^d \,\,\,\text{and}\,\,\, u-g\in H^1_0(D) \Big\}.$$
Then, fixed $u\in\mathcal V$, we define the admissible set $\mathcal A(u)$ as the set of all couples $(\Omega_1,\Omega_2)$ of Lebesgue measurable sets such that: 
\begin{itemize}
\item $\Omega_1\cap\Omega_2=\emptyset$, $E_1\subset \Omega_1$ and $E_2\subset \Omega_2$ Lebesgue almost-everywhere;
\item $\Omega_1$ and $\Omega_2$ have finite perimeter (as subsets of $\R^d$);
\item $\{u>0\}\subset{\Omega_1\cup\Omega_2}$ Lebesgue almost-everywhere.
\end{itemize}	
For every $\beta>0$ and $\Lambda>0$, we consider the functional 
$J_{\beta,\Lambda}$,
defined for functions $u\in\mathcal V$ and couples of sets $(\Omega_1,\Omega_2)\in\mathcal A(u)$, as
$$J_{\beta,\Lambda}(u,\Omega_1,\Omega_2):= \int_D |\nabla u|^2\,dx+ \beta\int_{\partial^\ast\Omega_1\cap \partial^\ast \Omega_2}u^2 \,d\HH^{d-1}+\Lambda|\{u>0\}\cap D|,$$
where $\partial^\ast\Omega_j$ is the reduced boundary of $\Omega_j$; we recall that since $u$ is a bounded Sobolev function, the second integral is well-defined (see \cref{s:preliminaries}). \medskip

\noindent In this paper we consider the variational problem 
\begin{equation} \label{problema}
\min \Big\{J_{\beta,\Lambda}(u,\Omega_1, \Omega_2) \,:\, u\in \V\,,\ (\Omega_1,\Omega_2)\in \A(u)\Big\}.
\end{equation}
The main result is the following. 
\begin{teo}\label{t:main}
Let $D$ be a smooth bounded open set in $\R^d$, $d\ge 2$. Given sets $E_1$ and $E_2$, and a function $g$ as above, there are a function $u\in\mathcal V$ and sets  $(\Omega_1,\Omega_2)\in\mathcal A(u)$ that solve the variational problem \eqref{problema}. Conversely, if $(u,\Omega_1,\Omega_2)$ is a solution to \eqref{problema}, then also $(u,\widetilde \Omega_1,\widetilde\Omega_2)$ is a solution to \eqref{problema}, where
$$\widetilde\Omega_1=\{u>0\}\cap\Omega_1\qquad\text{and}\qquad\widetilde\Omega_2=\{u>0\}\cap\Omega_2.$$
Moreover,
\begin{enumerate}[\quad\rm(i)]
\item the boundary $\partial\{u>0\}$ is $C^{1,\alpha}$-regular in $D$;\smallskip
\item the interface $\partial\Omega_1\cap\partial\Omega_2$ is $C^\infty$ in the open set  $D\cap\{u>0\}$ and is $C^{1}$ regular up to the boundary $D\cap\partial\{u>0\}$. Moreover,  $\partial\Omega_1\cap\partial\Omega_2$ reaches $\partial\{u>0\}$ orthogonally.
\end{enumerate}	
\end{teo}	 

\subsection*{Sketch of the proof and plan of the paper} In order to prove \cref{t:main}, we first introduce a family of approximating problems in \cref{s:existence-approx}. Then, passing to the limit, we obtain a function $u\in\mathcal V$ and a couple of disjoint sets $\Omega_1$ and $\Omega_2$. We cannot obtain immediately that $(u,\Omega_1,\Omega_2)$ is a solution to \eqref{problema}, since there is not a uniform bound on the perimeter of the approximating sets, so we do not a priori have that $\Omega_1$ and $\Omega_2$ are sets of locally finite perimeter in $D$. Instead, we are able to prove that $u$ satisfies an almost-minimality condition involving the one-phase Alt-Caffarelli functional, which allows to prove that the set $\{u>0\}$ is regular (\cref{t:regularity-free-boundary}). This solves the problem only in part because we only have that 
$${\{u>0\}}={\Omega_1\cup\Omega_2}\quad\text{in}\quad D.$$
We then show that the sets $\Omega_1$ and $\Omega_2$ are almost-minimizers of the perimeter in $\{u>0\}\cap D$, which implies that (in low dimension) the free interface $\partial\Omega_1\cup\partial\Omega_2$ is smooth in $D\cap\{u>0\}$. Thus, in order to prove that $\Omega_1$ and $\Omega_2$ have finite perimeter it is sufficient to study the behavior of the interface $\partial\Omega_1\cup\partial\Omega_2$ close to the free boundary $\partial\{u>0\}$ (see \cref{t:regularity-free-interface}). We show that $\Omega_1$ and $\Omega_2$ are minimizers in $\{u>0\}$ of a weighted perimeter functional, the weight being precisely the function $u^2$, which is $C^{0,\alpha}$ and positive in $\{u>0\}$, but as it approaches the free boundary $\partial\{u>0\}$ we have that 
$$u^2(x)\sim \text{dist}^2\big(x,\partial\{u>0\}\big).$$ 
In order to deal with this degenerate weight, we perform a 2D conformal change of coordinates, which flattens $\partial\{u>0\}$ to a line; then we rotate $\Omega_1$ around this line in order to obtain an almost-minimizer of the perimeter in $\R^4$. This allows to conclude that $\partial\Omega_1\cup\partial\Omega_2$ is the union of $C^{1}$ curves that meet $\partial\{u>0\}$ orthogonally in a (locally) finite number of points. This concludes the proof of \cref{t:regularity-free-interface} and also shows that $\Omega_1$ and $\Omega_2$ have locally finite perimeter. In order to show that $\Omega_1$ and $\Omega_2$ have finite perimeter, in \cref{p:non-collapsing} we prove that $\{u>0\}$ contains strictly both $\overline E_1$ and $\overline E_2$. Then, in \cref{s:proof12} we show that $\Omega_1$ and $\Omega_2$ are actually minimizers of \eqref{problema} and we complete the proof of \cref{t:main}.

\section{Sets of finite perimeter and Sobolev functions}\label{s:preliminaries}

\subsection{Caccioppoli sets}
For any Lebesgue measurable set $\Omega\subset\R^d$, we define
$$\text{\rm Per}(\Omega):=\sup\Big\{\int_{\Omega}\text{\rm div}\, \xi(x)\,dx\ :\ \xi\in C^1_c(\R^d),\ \|\xi\|_{L^\infty(\R^d)}\le 1\Big\},$$
and we say that $\Omega$ is of finite perimeter (Caccioppoli set) if 
$$\text{\rm Per}\,(\Omega)<+\infty\,.$$
Given $\alpha\in[0,1]$, we say that the set $\Omega$ has a Lebesgue density $\alpha$ at $x\in\R^d$ if 
$$\lim_{r\to0}\frac{|\Omega\cap B_r(x)|}{|B_r|}=\alpha\,.$$
We define the set $\Omega^{(\alpha)}$ as 
$$\Omega^{(\alpha)}:=\Big\{x\in \R^d\ :\ \lim_{r\to0}\frac{|\Omega\cap B_r(x)|}{|B_r|}=\alpha \Big\}.$$
Given a set of finite perimeter $\Omega\subset\R^d$ we will denote by $\partial^\ast\Omega$ its reduced boundary and by $\nu_\Omega$ the generalized exterior normal. We recall that
$$\text{Per}(\Omega)=\HH^{d-1}(\partial^\ast\Omega),$$
and that for any $\xi\in C^1_c(\R^d)$
$$\int_{\Omega}\text{\rm div}\, \xi(x)\,dx=\int_{\partial^\ast\Omega} \xi\cdot\nu_\Omega\,d\HH^{d-1},$$
where $\HH^{d-1}$ denotes the $(d-1)$-dimensional Hausdorff measure in $\R^d$. Moreover, we recall that at every point of the reduced boundary, $\partial^\ast\Omega$ has Lebesgue density $\sfrac12$, that is,
$$\partial^\ast\Omega\subset\Omega^{(\sfrac12)}.$$
We also recall the following well-known result by Federer
$$\HH^{d-1}\Big(\R^d\setminus\big\{\Omega^{(0)}\cup\Omega^{(1)}\cup\partial^\ast\Omega\big\}\Big)=0,$$
which can also be stated as in the lemma below.
\begin{lemma}\label{l:federer}
If $\Omega$ is a set of finite perimeter in $\R^d$, then up to a set of zero $\HH^{d-1}$ measure
$$\Omega^{(\sfrac12)}=\partial^\ast\Omega\qquad\text{and}\qquad \Omega^{(0)}\cup \Omega^{(\sfrac12)}\cup \Omega^{(1)}=\R^d.$$	
\end{lemma}	 

Finally, we conclude this section with the following proposition 
\begin{prop}
Let $A$ and $B$ be two disoint sets of finite perimeter in $\R^d$. Then,
\begin{align*}
\partial^\ast A=\Big(\partial^\ast A\cap \partial^\ast B\Big)\cup \Big(\partial^\ast A\setminus \partial^\ast B\Big)\qquad\text{and}\qquad 
\partial^\ast B=\Big(\partial^\ast A\cap \partial^\ast B\Big)\cup \Big(\partial^\ast B\setminus \partial^\ast A\Big),
\end{align*}
the set $A\cup B$ is a set of finite perimeter and, up to a set of zero $\HH^{d-1}$-measure, and
\begin{align}
\partial^\ast(A\cup B)=\Big(\partial^\ast A\setminus \partial^\ast B\Big)\cup\Big(\partial^\ast B\setminus \partial^\ast A\Big).\label{e:perimeterAB}
\end{align}
In particular,
\begin{align}
\text{\rm Per}(A)+\text{\rm Per}(B)=\text{\rm Per}(A\cup B)+2\,\HH^{d-1}\big(\partial^\ast A\cap \partial^\ast B\big).\label{e:perimeter-decomposition}
\end{align}	
\end{prop}
\begin{proof}
Up to a set of zero $\HH^{d-1}$ measure, we have that 
$$\partial^\ast A\setminus \partial^\ast B=A^{(\sfrac12)}\cap\Big(B^{(0)}\cup B^{(1)}\Big)=A^{(\sfrac12)}\cap B^{(0)}.$$
Analogously, $\partial^\ast B\setminus \partial^\ast A=B^{(\sfrac12)}\cap A^{(0)}.$
On the other hand
$$\partial^{\ast}(A\cup B)=(A\cup B)^{(\sfrac12)}=\Big(A^{(\sfrac12)}\cap B^{(0)}\Big)\cup\Big(B^{(\sfrac12)}\cap A^{(0)}\Big),$$
which proves \eqref{e:perimeterAB}. Finally, \eqref{e:perimeter-decomposition} follows since the sets
$$\partial^\ast A\setminus \partial^\ast B\ ,\quad \partial^\ast B\setminus \partial^\ast A\quad\text{and}\quad \partial^\ast A\cap \partial^\ast B\ ,$$ 
are disjoint.
\end{proof}		

As a consequence of \cref{l:federer}, one can obtain the following decomposition.
\begin{prop}\label{p:decompositions}
Let $A$ and $B$ be two sets of finite perimeter in $\R^d$. Then, also $A\setminus B$ and $B\setminus A$ have finite perimeter and we have the following decompositions (up to sets of zero $\HH^{d-1}$ measure) of $\partial^\ast A$, $\partial^\ast B$, $\partial^\ast(A\setminus B)$ and $\partial^\ast(B\setminus A)$ into disjoint sets:
	\begin{align*}
	\partial^\ast A&=\Big(A^{(\sfrac12)}\cap B^{(0)}\Big)\cup \Big(A^{(\sfrac12)}\cap B^{(1)}\Big)\cup \Big(A^{(\sfrac12)}\cap B^{(\sfrac12)}\Big)\\
	\partial^\ast B&=\Big(B^{(\sfrac12)}\cap A^{(0)}\Big)\cup \Big(B^{(\sfrac12)}\cap A^{(1)}\Big)\cup \Big(B^{(\sfrac12)}\cap A^{(\sfrac12)}\Big)\\
	\partial^\ast (A\setminus B)&=\Big(A^{(1)}\cap B^{(\sfrac12)}\Big)\cup \Big(A^{(\sfrac12)}\cap B^{(0)}\Big)\cup \Big(A^{(\sfrac12)}\cap B^{(\sfrac12)}\cap (A\cup B)^{(1)}\Big)\\
	\partial^\ast (B\setminus A)&=\Big(B^{(1)}\cap A^{(\sfrac12)}\Big)\cup \Big(B^{(\sfrac12)}\cap A^{(0)}\Big)\cup \Big(B^{(\sfrac12)}\cap A^{(\sfrac12)}\cap (A\cup B)^{(1)}\Big).
	\end{align*}
\end{prop}

\subsection{Sobolev functions and capacity}
Let $u:\R^d\to\R$ be a measurable function. We recall that $u$ is a Sobolev function ($u\in H^1(\R^d)$), if $u\in L^2(\R^d)$ and $\nabla u\in L^2(\R^d;\R^d)$, where $\nabla u$ is the distributional gradient of $u$. Given a measurable set $\Omega\subset\R^d$ and a Sobolev function $u\in H^1(\R^d)$, we say that $u\in \widetilde H^1_0(\Omega)$ if 
$$u=0\quad\text{almost everywhere on}\quad \R^d\setminus \Omega.$$
If $\Omega$ is an open set, we can also define the space $H^1_0(\Omega)$ as the closure of $C^\infty_c(\Omega)$ with respect to the Sobolev norm $$\|u\|_{H^1}:=\Big(\|u\|_{L^2}^2+\|\nabla u\|_{L^2}^2\Big)^{\sfrac12}.$$
It is well-known that both $H^1_0(\Omega)$ and $\widetilde H^1_0(\Omega)$ are closed (with respect to both the strong and the weak $H^1$-convergence) linear subspaces of $H^1(\R^d)$ and that, for any open set $\Omega$, $H^1_0(\Omega)\subset\widetilde H^1_0(\Omega)$, while the converse inclusion is in general false.  \medskip

Given any set $A\subset\R^d$ and any ball $B_{2R}(x_0)$, we define 
\begin{align*}
\cp\Big(A;B_{2R}(x_0)\Big):=\inf\Big\{&\int|\nabla\varphi|^2\,dx\ :\ \varphi\in H^1_0(B_{2R}(x_0)),\\
&\qquad \varphi\ge 1\ \text{in a neighborhood of}\  B_{R}(x_0)\cap A\Big\}.
\end{align*}
We say that a set $A$ has zero capacity if 
$$\cp\Big(A;B_{2R}(x_0)\Big)=0\qquad\text{for every ball}\qquad B_{2R}(x_0)\subset\R^d.$$
We recall the following properties of the capacity.
\begin{itemize}
\item If a set $A\subset\R^d$ has zero capacity, then $|A|=0$ and $\HH^{d-1}(A)=0$.
\item Given $u\in H^1(\R^d)$, there exists a set of zero capacity $\mathcal N_u$ such that 
$$\lim_{r\to0}\frac{1}{|B_r(x_0)|}\int_{B_r(x_0)}u(x)\,dx\quad\text{exists}\quad  \text{for every}\quad x_0\in\R^d\setminus\mathcal N_u\,.$$
In particular, to every $u\in H^1(\R^d)$, we can associate a representative
$$\widetilde u:\R^d\to\R\,$$ 
defined pointwise everywhere as follows:
$$\widetilde u(x_0):=\lim_{r\to0}\frac{1}{|B_r(x_0)|}\int_{B_r(x_0)}u(x)\,dx\quad\text{if}\quad x_0\in\R^d\setminus\mathcal N_u\ ,$$
while $\widetilde u(x_0)=0$ if $x_0\in\mathcal N_u$.
\item Suppose that a sequence $u_n\in H^1(\R^d)$ converges strongly in $H^1$ to $u\in H^1(\R^d)$. Let $\widetilde u_n$ and $\widetilde u$ be the representatives defined above and let $\mathcal N_{u_n}$ and $\mathcal N_u$ be the correspondig sets of zero capacity. Then, there are a subsequence $u_{n_k}$ and a set of zero capacity $\mathcal N$ such that 
$$\mathcal N_u\subset\mathcal N\qquad\text{and}\qquad\bigcup_{n\ge 1}\mathcal N_{u_n}\subset\mathcal N,$$
and 
\begin{equation}\label{e:qe-convergence}
\lim_{k\to\infty}\widetilde u_{n_k}(x)=\widetilde u(x)\quad\text{for every}\quad x\in\R^d\setminus \mathcal N.
\end{equation}
For simplicity, we will identify any function $u\in H^1(\R^d)$ with its representative $\widetilde u$ and if a sequence $u_{n_k}\in H^1(\R^d)$ satisfies \eqref{e:qe-convergence}, then we will say that it converges quasi-everywhere to $u\in H^1(\R^d)$.
\end{itemize}	

\subsection{Traces of Sobolev functions on the boundary of sets of finite perimeter}
Let $\Omega$ be a set of finite perimeter in $\R^d$ and let $u\in H^1(\R^d)$. Let $\widetilde u:\R^d\to\R$ be the representative of $u$ defined for every point $x_0$ outside a set of zero capacity $\mathcal N_u$ (and defined as zero on $\mathcal N_u$). Then, $\widetilde u$ is defined at every point of $\partial^\ast\Omega\setminus\mathcal N_u$. Since $\mathcal N_u$ has zero $\HH^{d-1}$-measure, we have that $\widetilde u$ is defined $\HH^{d-1}$-almost everywhere on $\partial^\ast\Omega$. We also notice that
$$\widetilde u:\partial^\ast\Omega\to\R$$
is a $\HH^{d-1}$ measurable function. Indeed, since $u\in H^1(\R^d)$ is a strong $H^1$ limit of a  sequence of $C^\infty$ functions, we have that $\widetilde u:\partial^\ast\Omega\to\R$ is a pointwise limit of smooth functions. From now on, we will write $u$ instead of $\widetilde u$. 

The next two propositions allow to write the functional $J_{\beta,\Lambda}$ in an equivalent way.

\begin{prop}\label{l:u=0}
Let	$\Omega\subset\R^d$ be a bounded quasi-open set of finite perimeter and let $\partial^\ast\Omega$ be its reduced boundary. Let $u\in \widetilde H^1_0(\Omega)$ and let $\widetilde u:\R^d\to\R$ be a representative of $u$ defined up to a set of zero capacity. Then, 
$$\widetilde u=0\quad\text{$\HH^{d-1}$-almost everywhere on}\quad \Omega^{(\sfrac12)}.$$
In particular, 
$$\widetilde u=0\quad\text{$\HH^{d-1}$-almost everywhere on}\quad \partial^\ast\Omega.$$
\end{prop}	
\begin{proof}
Without loss of generality, we can suppose that 
$$0\le u\le 1.$$
For every $n\ge 1$, we consider the functional 
$$F_n:H^1_0(\Omega)\to\R\ ,\qquad F_n(v)=\int_{\Omega}|\nabla v|^2\,dx+n\int_{\Omega}|v-u|^2\,dx.$$
The functional $F_n$ admits a unique minimizer in $H^1_0(\Omega)$ that we denote by $u_n$. By construction, testing the optimality of $u_n$ with $v=u$, we get 
$$\int_{\Omega}|\nabla u_n|^2\,dx+n\int_{\Omega}|u_n-u|^2\,dx\le \int_{\Omega}|\nabla u|^2\,dx.$$
In particular, the sequence $u_n$ converges strongly $L^2(\Omega)$ and weakly in $H^1_0(\Omega)$ to the function $u$. Moreover, $u_n$ solves the PDE 
\begin{equation}\label{e:u_n}
-\Delta u_n=n(u-u_n)\quad\text{in}\quad\Omega\ ,\qquad u\in H^1_0(\Omega).
\end{equation}
We notice that since $u_n$	minimizes $F_n$ and since $0\le u\le 1$, then also 
$$0\le u_n\le 1.$$
Thus, the right-hand side $n(u-u_n)$ of \eqref{e:u_n} is bounded. Let now $x_0\in\Omega^{(\sfrac12)}$, that is,  
$$\lim_{r\to0}\frac{|B_r(x_0)\cap \Omega|}{|B_r(x_0)|}=\frac12.$$
By \cite[Proposition 4.6]{DV} we have that for $r>0$ small enough
$$\|u_n\|_{L^\infty(B_r(x_0)}\le r^\beta C_n,$$
for some constant $C_n$ depending on $u_n$ and some dimensional constant $\beta>0.$
In particular, 
$$u_n=0\quad\text{on}\quad \Omega^{(\sfrac12)}.$$
Now, since 
$$\int_{\Omega}|\nabla u_n|^2\,dx\le \int_{\Omega}|\nabla u|^2\,dx\qquad\text{for every}\qquad n\ge 1\ ,$$
we get that the convergence of $u_n$ to $u$ is strong in $H^1_0(\Omega)$. It is well-known that there is a subsequence of $u_n$ converging pointwise quasi-everywhere to $u$. In particular, the same subsequence converges pointwise $\HH^{d-1}$-almost
everywhere on $\partial^\ast\Omega$. Thus, (the representative of) $u$ vanishes $\HH^{d-1}$-almost everywhere on both $\Omega^{(\sfrac12)}$ and $\partial^\ast\Omega$.
\end{proof}

\begin{prop}
Let $\Omega_1$ and $\Omega_2$ be two sets of finite perimeter in $\R^d$ such that 
$$|\Omega_1\cap\Omega_2|=0.$$
Then, for any function $u\in \widetilde H^1_0(\Omega_1\cup\Omega_2)\cap L^\infty(\Omega_1\cup\Omega_2)$, we have that 
\begin{equation}\label{e:cannoli1}
\int_{\partial^\ast\Omega_1\cap \partial^\ast \Omega_2}u^2\,d\HH^{d-1}= \frac12\left(	\int_{\partial^\ast\Omega_1}u^2\, d\HH^{d-1}+\int_{\partial^\ast\Omega_2}u^2\, d\HH^{d-1}\right).
\end{equation}	
\end{prop}
\begin{proof}
We first notice that the reduced boundaries $\partial^\ast\Omega_1$ and $\partial^\ast\Omega_2$ can be decomposed as
$$\partial^\ast\Omega_1=\Big(\partial^\ast\Omega_1\cap\partial^\ast\Omega_2\Big)\cup \Big(\partial^\ast\Omega_1\setminus\partial^\ast\Omega_2\Big)\quad\text{and}\quad\partial^\ast\Omega_2=\Big(\partial^\ast\Omega_1\cap\partial^\ast\Omega_2\Big)\cup \Big(\partial^\ast\Omega_2\setminus\partial^\ast\Omega_1\Big).$$
Thus, in order to prove \eqref{e:cannoli1}, it is sufficient to prove that 
\begin{equation}\label{e:cannoli2}
u=0\qquad\text{$\HH^{d-1}$-almost everywhere on}\qquad \big(\partial^\ast\Omega_1\setminus \partial^\ast \Omega_2\big)\cup \big(\partial^\ast \Omega_2\setminus \partial^\ast \Omega_1\big).
\end{equation}
Let $x_0\in \partial^\ast\Omega_1\setminus \partial^\ast \Omega_2$. Since $$\HH^{d-1}\Big(\R^d\setminus\big(\Omega_2^{(0)}\cup \Omega_2^{(1)}\cup\partial^\ast\Omega_2\big)\Big)=0,$$ 
we can suppose that $x_0\in \Omega_2^{(0)}\cup \Omega_2^{(1)}$, but since $x_0\in\Omega_1^{(\sfrac12)}$ and $\Omega_1\cap\Omega_2=0$, we get that necessarily $x_0\in\Omega_2^{(0)}$. But then, 
$$x_0\in (\Omega_1\cup\Omega_2)^{(\sfrac12)}.$$
Thus, by Proposition \ref{l:u=0}, we get that $u(x_0)=0$. This proves \eqref{e:cannoli2} and \eqref{e:cannoli1}.
\end{proof}

\subsection{A semicontinuity lemma}

In the proof of the main theorem we will repeatedly use the following lemma, which is a restatement of a lemma from \cite{GLMV}.
\begin{lemma}[\cite{GLMV}]\label{l:semicontinuity}
	Let $A\subset\R^d$ be a bounded open set. Let $u_n\in H^1(A)$ be a sequence of functions converging to $u_\infty\in H^1(A)$ weakly in $H^1(A)$, strongly in $L^2(A)$ and pointwise almost-everywhere. Let $\Omega_n\subset A$ be a sequence of sets of locally finite perimeter in $A$ converging almost-everywhere (in $A$) to the set of locally finite perimeter $\Omega_\infty\subset A$.
	Then, 
	\begin{equation}\label{e:semicontinuity}
	\int_{A\cap\partial^\ast\Omega_\infty}u_\infty^2\,d\HH^{d-1}\le \liminf_{n\to\infty} \int_{A\cap\partial^\ast\Omega_n}u_n^2\,d\HH^{d-1}.
	\end{equation}
\end{lemma}

\begin{proof}	
	The proof is precisely the one from \cite[Lemma 2.4]{GLMV}.	We report it here for the sake of completeness. The key observation is that given $u\in H^1(A)$ and a set of locally finite perimeter $\Omega\subset A$, we have 
	$$\int_{A\cap\partial^\ast \Omega}u^2\,d\HH^{d-1}=\sup\left\{\int_{A\cap \Omega}\text{\rm div}\,\big(u^2\xi\big)\,dx\ :\ \xi\in C^1_c(A;\R^d),\ |\xi|\le 1\right\}.$$
	We now fix a vector field $ \xi\in C^1_c(A;\R^d),\ |\xi|\le 1$ and we compute 
	\begin{align*}
	\liminf_{n\to\infty}\int_{A\cap\partial^\ast \Omega_n}u_n^2\,d\HH^{d-1}&\ge	
	\liminf_{n\to\infty}\int_{A\cap\Omega_n}\text{div}\big(u_n^2\xi\big)\,dx\\
	&= 	\liminf_{n\to\infty}\int_{A}\Big(2\big(u_n\xi\big)\cdot \big(\ind_{\Omega_n}\nabla u_n\big)+\big(u_n\,\ind_{\Omega_n}\big)\,(u_n\,\text{div}\,\xi)\Big)\,dx\,
	\end{align*}
	Now, since $\ind_{\Omega_n}\nabla u_n$ converges weakly in $L^2$ to $\ind_{\Omega_\infty}\nabla u_\infty$, we get
	\begin{align*}
	\liminf_{n\to\infty}\int_{A\cap\partial^\ast \Omega_n}u_n^2\,d\HH^{d-1}&\ge 	\int_{A}\Big(2\big(u_\infty\xi\big)\cdot \big(\ind_{\Omega_\infty}\nabla u_\infty\big)+\big(u_\infty\,\ind_{\Omega_\infty}\big)\,(u_\infty\,\text{div}\,\xi)\Big)\,dx\,\\
	&=\int_{A\cap\Omega_\infty}\text{div}\big(u_\infty^2\xi\big)\,dx.
	\end{align*}
	Taking the supremum over $\xi$, we get \eqref{e:semicontinuity}. 
\end{proof}	

\section{Almost-minimality and H\"older estimates}\label{s:holder}
In this section, we prove two general technical results on the continuity of subharmonic functions which are almost-minimizers of the Dirichlet energy in a suitable sense. We will use these estimates in Sections \ref{s:existence-approx} and \ref{s:regularity}.
\begin{lemma}[A growth estimate]\label{l:holder1}
	Let $D$ be a bounded open set in $\R^d$, $x_0\in D$ and let $u\in H^1(D)$ be a function such that 
	\begin{enumerate}[\quad\rm(a)]
		\item $u$ is non-negative and subharmonic in $D$; 
		\item there are constants $\alpha\in[0,1]$, $K>0$ and $r_0>0$ such that $B_{r_0}(x_0)\subset D$ and
		$$\int_{B_r(x_0)}|\nabla u|^2\,dx\le \int_{B_r(x_0)}|\nabla (u+\varphi)|^2\,dx+Kr^{d-1+\alpha},$$
		for every $r\in(0,r_0)$ and every $\varphi\in H^1_0(B_{r}(x_0))$ with $\varphi\ge 0$ in $B_r(x_0)$.
	\end{enumerate}	 
	Then, for every $r\le r_0/4$, we have 
$$\mean{B_r(x_0)}{u(x)\,dx}-u(x_0)\le \frac{C_d\sqrt{K}}{\alpha+1}\,{r^{\frac{1+\alpha}2}}.$$
In particular, if $u(x_0)=0$, then 
	$$\|u\|_{L^\infty(B_r(x_0))}\le \frac{C_d\sqrt{K}}{\alpha+1}\,{r^{\frac{1+\alpha}2}}\quad\text{for every}\quad r\le \frac{r_0}{8}.$$
\end{lemma}	
\begin{proof}
Without loss of generality, we can suppose that $x_0=0$ and $D=B_{r_0}$. Then, for every $r\in(0,r_0)$, we have 
$$0\le \mean{\partial B_r}{u\,d\HH^{d-1}}-u(0)\le \frac{1}{d\omega_d}\int_{0}^r{s^{1-d}}\Delta u(B_s)\,ds.$$
Thus, we only need to estimate $\Delta u(B_r)$. In order to do so, we test the optimality of $u$ with $u+t\varphi$, where $\varphi(x):=\frac{1}{r}(2r-|x|)_+\,.$  
$$-2\int_{B_r}\nabla u\cdot\nabla\varphi\,dx\le t\int_{B_r}|\nabla\varphi|^2+\frac{K}{t}r^{d-1+\alpha}\le t\omega_dr^{d-2}+\frac{K}{t}r^{d-1+\alpha}.$$
Taking 
$$t:=\sqrt{\frac{K}{\omega_d}}\,r^{\frac{1+\alpha}{2}},$$
we obtain
$$-2\int_{B_r}\nabla u\cdot\nabla\varphi\,dx\le \sqrt{{K}{\omega_d}}\,r^{d-\frac32+\frac{\alpha}{2}},$$
and so, for every $r\le \frac{r_0}2$, we get 
\begin{align*}
\Delta u(B_{r/2})\le \int_{B_{r}}\varphi(x)\Delta u(x)\,dx\le -\int_{B_{r}}\nabla\varphi\cdot\nabla u\,dx\le \frac12\sqrt{{K}{\omega_d}}\,r^{d-\frac32+\frac{\alpha}{2}}.
\end{align*}
Then, for every $r\le\frac{r_0}{4}$, we can estimate 
$$\mean{\partial B_r}{u\,d\HH^{d-1}}-u(0)\le \frac{2^{d-1}\sqrt{K\omega_d}}{d\omega_d}\int_{0}^r{s^{1-d}}s^{d-\frac32+\frac{\alpha}{2}}\,ds\le \frac{2^{d}\sqrt{K\omega_d}}{d\omega_d}\frac{r^{\frac{1+\alpha}2}}{\alpha+1}.$$
Now, using the non-negativity and the subharmonicity of $u$, we get the claim.
\end{proof}	

\begin{lemma}[H\"older continuity]\label{l:holder2}
	Let $D$ be a bounded open set in $\R^d$ let $u\in H^1(D)$ be a function which is subharmonic in $D$ and such that
	$$0\le u\le M\quad\text{in}\quad D,$$ 
	for some constant $M>0$. Given $\delta\in(0,1)$, we define the set 
		\begin{equation}\label{e:Ddelta}
	D_\delta:=\big\{x\in D\ :\ \text{\rm dist}(x,\partial D)> \delta\big\},
	\end{equation}
	and we suppose that there are constants $\alpha\in[0,1]$ and $K>0$ such that 
		$$\int_{B_r(x_0)}|\nabla u|^2\,dx\le \int_{B_r(x_0)}|\nabla (u+\varphi)|^2\,dx+Kr^{d-1+\alpha},$$
		for every $x_0\in D_\delta$, every $r\in(0,\delta)$ and every $\varphi\in H^1_0(B_{r}(x_0))$ with $\varphi\ge 0$ in $B_r(x_0)$.
	Then, 
	$$|u(x)-u(y)|\le C_d\left(M+\frac{\sqrt{K}}{\alpha+1}\right)\,{|x-y|^{\frac{1+\alpha}{3+\alpha}}},$$
	for every $x,y\in D_\delta$ such that $|x-y|<\big(\frac\delta{16}\big)^{2}$
\end{lemma}	
\begin{proof}
We apply the previous lemma, to $x_0=y$ and $r=|x-y|^\gamma$. Then, 
\begin{align*}
u(x)-u(y)&\le \mean{B_r(x)}u-u(y)\le \frac{(r+|x-y|)^d}{r^d}\mean{B_{r+|x-y|}(y)}u-u(y)\\
&\le \frac{(r+|x-y|)^d-r^d}{r^d}M+\mean{B_{r+|x-y|}(y)}u-u(y)\\
&\le \frac{|x-y|}{r}Md\frac{(r+|x-y|)^{d-1}}{r^{d-1}}+\frac{C_d\sqrt{K}}{\alpha+1}\,{(r+|x-y|)^{\frac{1+\alpha}2}}.
\end{align*}
Now, since $|x-y|<1$ and $r=|x-y|^\gamma$, with $\gamma\in(0,1)$, we get 
\begin{align*}
u(x)-u(y)&\le d2^dM|x-y|^{1-\gamma}+\frac{2C_d\sqrt{K}}{\alpha+1}\,|x-y|^{\gamma\frac{1+\alpha}2}.
\end{align*}
Choosing $\gamma=\frac{2}{3+\alpha}$, we get that $$1-\gamma=\gamma\frac{1+\alpha}{2}=\frac{\alpha+1}{\alpha+3},$$
and so we get the claim. Finally, we notice that we should have the inequality $r+|x-y|<\frac\delta8$, which is satisfied for instance when $|x-y|^\gamma<\frac\delta{16}$.
\end{proof}

\section{Non-degenerate approximating problems}\label{s:existence-approx}
In this section, we define a sequence of non-degenerate problems, approximating \eqref{problema}, in which the competitors $u$ are a priori bounded from below by a fixed positive constant. \medskip
%
%

\noindent We consider the family of approximating problems
\begin{equation} \label{e:approx}
\min\Big\{J_\eps\big(u,\Omega_1, \Omega_2\big) \,:\ u\in\mathcal V,\ \big(\Omega_1,\Omega_2\big)\in \A(u)\Big\},
\end{equation}
where the functional $J_\eps$ is defined as
\begin{align*}
J_\eps(u,\Omega_1,\Omega_2)&:= \int_D |\nabla u|^2\,dx+  \Lambda|\{u>0\}\cap D|\\
&\qquad+\frac\beta2\left(	\int_{\partial^\ast\Omega_1}u^2\, d\HH^{d-1}+\int_{\partial^\ast\Omega_2}u^2\, d\HH^{d-1}\right)+\eps\Big(\text{\rm Per}(\Omega_1)+\text{\rm Per}(\Omega_2)\Big).
\end{align*}

\begin{prop}[Solutions to the approximating problem]\label{p:approx}
Let $D$ be a bounded open set in $\R^d$. Then, for any $\eps>0$ there are  $u_\eps\in \V$ and  $(\Omega^1_\eps,\Omega^2_\eps)\in\mathcal A(u_\eps)$ such that:
\begin{enumerate}[\quad\rm(i)]
\item $(u_\eps,\Omega^1_\eps,\Omega^2_\eps)$ is a solution to \eqref{e:approx};
\item the function $u_\eps$ is H\"older continuous in $D$ and for every $\delta>0$ there is a constant $C_\delta>0$, depending on $\delta$, $d$, $\beta$, $\Lambda$ and $\|g\|_{L^\infty}$, such that
\begin{equation} \label{eq:holder1}
|u_\eps(x)-u_\eps(y)|\leq C_\delta |x-y|^{\sfrac13}\qquad\text{for every}\qquad x,y\in D_\delta,
\end{equation}
where $D_\delta$ is the set defined in \eqref{e:Ddelta}.
\item there is a constant $\rho>0$, depending only on $d$, $\Lambda$ and $\|g\|_{L^\infty}$, such that
\begin{equation} \label{e:compact-support-eps}
\Omega_\eps^1\cup\Omega_\eps^2\subset (E_1\cup E_2)+B_\rho.
\end{equation}
\end{enumerate}
\end{prop}
\begin{proof} 
	Let $\eps>0$ be fixed. We divide the proof into several steps.	\medskip
	
\noindent \textbf{Existence.}
Let $\big\{\big(u_{\eps,i},\Omega_{\eps,i}^1,\Omega_{\eps,i}^2\big)\big\}_{i\in \N}$ be a minimizing sequence for \eqref{e:approx}. Since we can use 
	$\big(g,E_1,E_2\big)$
	as a competitor against $\big(u_{\eps,i},\Omega_{\eps,i}^1,\Omega_{\eps,i}^2\big)$, we have that
	\begin{equation} \label{e:minimizingapp}
	\begin{array}{ll}
	\ds\int_D |\nabla u_{\eps,i}|^2 \, dx+ \eps\Big(\text{\rm Per}(\Omega^1_{\eps,i})+\text{\rm Per}(\Omega^2_{\eps,i})\Big)\le J_{\eps}(g,E_1,E_2)\le J_{1}(g,E_1,E_2).
	\end{array}
	\end{equation}
	Moreover, since $u_{\eps,i}-g\in H^1_0(D)$, by the Poincar\'e inequality with constant $C_D$, we have 
	\begin{align*}
	\|u_{\eps,i}\|_{L^2(D)} &\leq \|u_{\eps,i}-g\|_{L^2(D)}+\|g\|_{L^2(D)}\\
	& \leq C_{D}\|\nabla(u_{\eps,i}-g)\|_{L^2(D)}+\|g\|_{L^2(D)}.
	\end{align*}
	Thus, the sequence $\{u_{\eps,i}\}_{i\in \mathbb N}$ is bounded in $H^1(D)$. Hence, up to a subsequence, there exists a function $u_\eps\in H^1_0(\R^d)$ such that
	$ \nabla u_{\eps,i} \to \nabla u_\eps$ weakly in $L^{2}(\R^d)$ and $u_{\eps,i}\to u_\eps$ strongly in $L^2(\R^d)$ and pointwise almost-everywhere. In particular, the almost-everywhere convergence gives that 
	\begin{equation}\label{e:semicont-measure-eps}
	\ind_{\{u_\eps>0\}}\le\liminf_{i\to\infty}\ind_{\{u_{\eps,i}>0\}}.
	\end{equation}
	
Using again \eqref{e:minimizingapp}, we obtain
$$	\text{Per}\big(\Omega_{\eps,i}^1\big)+\text{Per}\big(\Omega_{\eps,i}^2\big)\le \frac1{\eps}J_{1}(g,E_1,E_2).$$
Hence, there are sets of finite perimeter $\Omega^1_\eps$ and $\Omega_\eps^2$ such that
$$\Omega_\eps^1\cap\Omega^2_\eps=\emptyset\ ,\qquad E_1\subset \Omega^1_\eps\qquad\text{and}\qquad E_2\subset \Omega^2_\eps,$$
and such that (up to a subsequence)
$$\ind_{\Omega_{\eps,i}^1}\to \ind_{\Omega_\eps^1}\qquad\text{and}\qquad \ind_{\Omega_{\eps,i}^2}\to \ind_{\Omega_\eps^2},$$ 
strongly in $L^1(\R^d)$ and pointwise almost-everywhere. Together with \eqref{e:semicont-measure-eps}, this implies that 
$$\{u_\eps>0\}\subset \Omega_\eps^1\cup\Omega_\eps^2\ ,$$
so $(\Omega_\eps^1,\Omega_\eps^2)\in \mathcal A(u_\eps)$.
Finally, the semicontinuity of $J_\eps$ (see \cref{l:semicontinuity}) gives that $(u_\eps,\Omega^1_\eps, \Omega^2_\eps)$ is a  solution to the variational problem \eqref{e:approx}.\medskip

	\noindent\textbf{Subharmonicity of $u_\eps$.} 	Let $v$ be a function in $H^1(D)$ such that 
	$$v\le u_\eps\quad\text{in}\quad D\qquad \text{and}\qquad v=u_\eps\quad\text{on}\quad\partial D\,.$$ 
	Then, testing the optimality of $(u_\eps,\Omega^1_\eps,\Omega^2_\eps)$ with $(v_+,\Omega^1_\eps,\Omega^2_\eps)$ and using the fact that 
	$$v_+\le u_\eps\quad\text{in}\quad D,$$
	we get
	\begin{align*}
	\int_D|\nabla v|^2\,dx&\ge \int_D|\nabla v_+|^2\,dx	\ge  \int_D|\nabla u_\eps|^2\,dx.
	\end{align*} 	
	In particular, if $\varphi$ is a nonnegative function compactly supported in $D$, then we can apply the above inequality to $v:=u-t\varphi$ for some $t>0$. Then , by sending $t$ to zero, we get that 
	$$-\int_D\nabla \varphi\cdot\nabla u_\eps\,dx\ge 0,$$
	which means that the distributional Laplacian $\Delta u_\eps$ is a positive Radon measure in $D$. \medskip

	\noindent\textbf{H\"older continuity of $u_\eps$.} Let $\delta\in(0,1)$ and let $D_\delta$ be given by \eqref{e:Ddelta}. 
	Let $x_0\in D_\delta$ and $0<R<\delta$.
	Let $\varphi\in H^1_0(B_R (x_0))$ be such that $\varphi\ge 0$ in $B_{R}(x_0)$. 
	We set 
	$$\widetilde \Omega_\eps^1:= \Omega_\eps^1\cup B_R (x_0)\qquad\text{and}\qquad
	\widetilde \Omega_\eps^2:= \Omega_\eps^2\setminus B_R(x_0),$$ 
	so that $(\widetilde \Omega_\eps^1,\widetilde \Omega_\eps^2) \in \mathcal A(u_\eps)$.
Using $(u_\eps+\varphi, \widetilde \Omega_\eps^1,\widetilde \Omega_\eps^2)$ to test the optimality of $(u_\eps,\Omega^1_\eps,\Omega_\eps^2)$, we get 
	\begin{align*} 
	\int_{B_R(x_0)}&|\nabla u_\eps|^2\,dx+\Lambda|B_R(x_0)\cap\{u_\eps>0\}|+\frac\beta2\int_{\partial^\ast \Omega_\eps^1} u_\eps^2\, d\HH^{d-1}+\frac\beta2\int_{\partial^\ast \Omega_\eps^2} u_\eps^2\, d\HH^{d-1}
	\\
	&\leq \int_{B_R(x_0)} |\nabla (u_\eps+\varphi)|^2\,dx +\Lambda|B_R|+\frac{\beta}2\int_{ \partial^\ast \widetilde \Omega_\eps^1}  u_\eps^2\, d\HH^{d-1}+\frac{\beta}2\int_{\partial^\ast \widetilde \Omega_\eps^2}  u_\eps^2\, d\HH^{d-1}+ \eps\HH^{d-1}(\partial B_R).
	\end{align*}
		Moreover, the definition of $(u+\varphi,\widetilde \Omega^1_\eps,\widetilde \Omega^2_\eps)$ gives that, for $j=1,2$,
	\begin{equation}\label{eq:laplmeas1H}
	\begin{split}
\int_{ \partial^\ast \widetilde \Omega_\eps^j} u_\eps^2 \,d\HH^{d-1}&\le \int_{ \partial^\ast \Omega_\eps^j} u_\eps^2 \,d\HH^{d-1}+\int_{ \partial B_R(x_0)} u_\eps^2 \,d\HH^{d-1}\\
&\le \int_{ \partial^\ast \Omega_\eps^j} u_\eps^2 \,d\HH^{d-1}
	+ d\omega_d R^{d-1}\|g\|_{L^\infty(D)}^2.
	\end{split}
	\end{equation}
	Combining these two estimates, we obtain
		\begin{align*} 
	\int_{B_R(x_0)}|\nabla u_\eps|^2\,dx&\le \int_{B_R(x_0)} |\nabla (u_\eps+\varphi)|^2\,dx +\Lambda|B_R|+\big(\eps+\beta\|g\|_{L^\infty(D)}^2\big)\,d\omega_d R^{d-1}\\
	&\le \int_{B_R(x_0)} |\nabla (u_\eps+\varphi)|^2\,dx +\Big(\omega_d\Lambda\delta+\big(1+\beta\|g\|_{L^\infty(D)}^2\big)\,d\omega_d\Big) R^{d-1}.
	\end{align*}
	By Lemma \ref{l:holder2} we get that 
	\begin{equation} \label{eq:holder1}
	|u_\eps(x)-u_\eps(y)|\leq C |x-y|^{\sfrac13}\qquad\text{for every}\qquad x,y\in D_\delta,
	\end{equation}
	such that $\ds|x-y|\le \frac{\delta^2}{256}$, where the constant $C$ depends only on $d$, $\Lambda$, $\beta$ and $\|g\|_{L^\infty}$.\medskip
	
		\noindent\textbf{Boundedness of $\Omega_\eps^1$ and $\Omega_\eps^2$.} We first notice that $u_\eps$ is a subsolution of the Alt-Caffarelli functional, that is, 
		$$\int_{\R^d}|\nabla u_\eps|^2\,dx+\Lambda|\{u_\eps>0\}|\le \int_{\R^d}|\nabla v|^2\,dx+\Lambda|\{v>0\}|,$$ 
		for every $v\in H^1(\R^d)$ such that $u-v\in H^1_0(D)$ and $0\le v\le u_\eps$ on $\R^d$. In particular, this implies (see for instance \cite{velectures}) that the set $\{u_\eps>0\}$ lies in a sufficiently large ball $B_\rho$. Now, since outside $\{u_\eps>0\}$ the functional $J_\eps$ only accounts for the perimeter of $\Omega_\eps^1$ and $\Omega_\eps^2$, we have that these sets should be contained in the convex envelope of $\{u_\eps>0\}$, which in particular gives \eqref{e:compact-support-eps}.	
\end{proof}

\section{The limit of the non-degenerate solutions}\label{sub:existence-conclusion}
In this section we define the function $u$ (\cref{sub:u_infty}) and the sets $(\Omega_1,\Omega_2)$ (\cref{sub:omega12}), which we will prove to be solutions to the initial problem \eqref{problema}. Throughout this section, for any $\eps>0$, we fix a solution $\big(u_\eps,\Omega_\eps^1,\Omega_\eps^2\big)$ of the approximating problem \eqref{e:approx} for $J_\eps$. 

\subsection{The limit function}\label{sub:u_infty}
It is immediate to check that, there is a function 
$$u\in H^1(\R^d),$$
and a sequence $\eps_n\to0$ such that 
\begin{itemize}
	\item for every fixed $\delta>0$, $u_{\eps_n}\to u$ uniformly in  $D_\delta$ as $n\to\infty$\,; \smallskip
	\item $u_{\eps_n} \to u$ strongly in $L^2(\R^d)$ and pointwise almost-everywhere in $\R^d$\,;\smallskip
	\item $\nabla u_{\eps_n} \to \nabla u$ weakly in $L^2(\R^d)$. 
\end{itemize}
By construction, we have  $u-g\in H^1_0(D)$, while \cref{p:approx} gives that
$$u\in H^1_0\big((E_1\cup E_2)+B_\rho\big)\qquad\text{and}\qquad u\in C^{0,\sfrac13}(\overline D_\delta)\quad\text{for every}\quad\delta>0.$$
Moreover,
$$ 0\le u\le \|g\|_{L^\infty(\R^d)}\qquad \text{and}\qquad \Delta u\ge 0\quad\text{in}\quad D.$$
\subsection{The limit sets}\label{sub:omega12}
We next construct the sets $\Omega_1$ and $\Omega_2$. Choose a ball 
$$B_R(x_0)\subset\subset D\cap \{u>0\}.$$ 
Then, there are $t>0$ and $\delta>0$ such that
$$\overline B_R(x_0)\subset  D_\delta \cap \{u\ge t\},$$
where $D_\delta$ is given by \eqref{e:Ddelta}. 
By the uniform convergence of $u_{\eps_n}$ to $u$ on $D_\delta$, we can find $n_0\ge 1$ such that, the following inequality holds for every $n\ge n_0$: 
$$u_{\eps_n}\geq \frac{t}2\quad\text{in}\quad \overline B_R(x_0).$$
Using this inequality and the optimality of $\big(u_{\eps_n}, \Omega^1_{\eps_n},\Omega^2_{\eps_n}\big)$, we can estimate
\begin{align*}
\frac{2}{\beta}J_\eps\big(g, E_1,E_2\big)\ge \frac{2}{\beta}J_\eps\big(u_{\eps_n}, \Omega^1_{\eps_n},\Omega^2_{\eps_n}\big)
&\ge \int_{B_R(x_0)\cap \partial^\ast \Omega^1_{\eps_n}}u_{\eps_n}^2\,d\HH^{d-1}+ \int_{B_R(x_0)\cap \partial^\ast \Omega_{\eps_n}^2} u_{\eps_n}^2\, d\HH^{d-1}
\\
 &\ge \frac{t^2}{4} \Big(\text{\rm Per}\big(\Omega_{\eps_n}^1;B_R(x_0)\big)+\text{\rm Per}\big(\Omega_{\eps_n}^2;B_R(x_0)\big)\Big).
\end{align*}
Thus, the sets
$$\Omega^1_{\eps_n}\cap B_R(x_0)\qquad\text{and}\qquad \Omega^2_{\eps_n}\cap B_R(x_0)$$ 
have uniformly bounded perimeter. In particular, up to a subsequence there are sets 
$$\Omega_{R,x_0}^1\cap B_R(x_0)\qquad\text{and}\qquad \Omega_{R,x_0}^2\cap B_R(x_0),$$
of finite perimeter and such that, as $n\to\infty$, 
\[
\ind_{\Omega^1_{\eps_n} \cap B_R(x_0)}\to \ind_{\Omega_{R,x_0}^1}\qquad\text{and}\qquad \ind_{\Omega^2_{\eps_n} \cap B_R(x_0)}\to \ind_{\Omega_{R,x_0}^2},
\]
pointwise almost-everywhere and strongly in $L^1(\R^d)$.
Thus, by a diagonal sequence argument, we can define the sets $\Omega_1$ and $\Omega_2$ as the union of $\Omega^1_{R,x_0}$ and $\Omega^2_{R,x_0}$ over all balls 
$$B_R(x_0)\subset\subset D\cap\{u>0\},$$ 
of radius $R\in\Q$ and center with rational coordinates $x_0\in\Q^d$,
$$\Omega_i:=E_i\cup\bigcup_{R,x_0}\Omega^i_{R,x_0}\qquad \text{for}\quad i=1,2. $$
By construction, $\Omega_1$ and $\Omega_2$ have locally finite perimeter in $D\cap \{u>0\}$ and satisfy 
$$E_i\subset\Omega_i\subset  \Big(\big(E_1\cup E_2\big)\cap B_\rho\Big),$$
where we recall that $D:=\R^d\setminus \big(\overline E_1\cup \overline E_2\big)$. 
Moreover, we still have the pointwise convergence of the corresponding characteristic functions, that is, for $i=1,2$,
$$\ind_{\Omega_i}=\lim_{n\to\infty} \ind_{\Omega^i_{\eps_n}\cap\{u>0\}}\quad \text{in}\quad \R^d,$$
which implies that, for almost every $x\in \R^d$,
$$\ind_{\Omega_1 \cap\, \Omega_2}(x)=\ind_{\Omega_1 }(x)\cdot\ind_{\Omega_2}(x)
=\lim_{n\to\infty}\Big(\ind_{\Omega^1_{\eps_n} }(x)\cdot\ind_{\Omega^2_{\eps_n}}(x)\cdot\ind_{\{u>0\}}\Big)=0,$$
the sets $\Omega_1$ and $\Omega_2$ are disjoint, $|\Omega_1 \cap \Omega_2|=0$. 
\begin{oss}
Notice that we do not have a priori that $\Omega_1$ and $\Omega_2$ have finite perimeter in $\R^d$, so at this stage they might not be in the admissible class $\mathcal A(u)$ defined in \cref{sub:statement}.
\end{oss}	

\section{Almost-minimality and Lipschitz estimates of $u$}
In this section we will show that $u$ is an almost-minimizer (in some suitable sense) of the classical one-phase functional of Alt and Caffarelli. From this we deduce the Lipschitz growth of $u$ on the boundary, which we will use in \cref{sub:blow-up} in order to deduce the convergence of the blow-up sequences of $u$.
\begin{lemma}\label{l:outwards}
	Let $u\in H^1(\R^d)$ be the function from \cref{sub:u_infty} and let $\overline B_r(x_0)\subset \R^d\setminus \overline E_2$. Suppose that $v\in H^1(\R^d)$ is a function such that $v-u\in H^1_0(B_r(x_0))$.
	Then,
	\begin{align*}
	\ds\int_{B_r(x_0)}|\nabla u|^2\,dx&+\Lambda|B_r(x_0)\cap\{u>0\}|\\
	&\le \int_{B_r(x_0)}|\nabla v|^2\,dx+\Lambda|B_r(x_0)\cap\{v>0\}|+\beta\int_{\partial B_r(x_0)}u^2d\HH^{d-1}.
	\end{align*}
\end{lemma}	
\begin{proof}
	We can suppose that $v\ge 0$. Then, testing the minimality of $\big(u_{\eps_n},\Omega_{\eps_n}^1,\Omega_{\eps_n}^2\big)$, with the function $v$ (which we can do since $v-g\in H^1_0(D)$) and the sets 
	$$\Omega_{\eps_n}^1\cup  B_r(x_0)\qquad\text{and}\qquad \Omega_{\eps_n}^2\setminus B_r(x_0),$$
	we obtain
	\begin{align*}
	\int_{D}|\nabla u_{\eps_n}|^2\,dx&+\Lambda|D\cap\{u_{\eps_n}>0\}|\\
	&\le \int_{D}|\nabla v|^2\,dx+\Lambda|D\cap\{v>0\}|+\beta\int_{\partial B_r(x_0)}u_{\eps_n}^2\,d\HH^{d-1}.
	\end{align*}
	Passing to the limit as $n\to\infty$, we get the claim.
\end{proof}	

\begin{lemma}\label{l:upgrowth}
Let $u\in H^1(\R^d)$ be the function defined in \cref{sub:u_infty}. For every $\delta\in(0,1)$ and every $x_0\in  D_\delta\cap \{u=0\}$, we have 
$$\|u\|_{L^\infty(B_r(x_0))}\le Cr\quad\text{for every}\qquad r<\frac{\delta}{256},$$
where $C$ is a constant depending only on $d$, $\beta$, $\Lambda$ and $\|g\|_{L^\infty}$.
\end{lemma}
\begin{proof}
We set for simplicity $M:=\|g\|_{L^\infty}$. By \cref{l:outwards}, we have that for every $r\in(0,\delta)$
\begin{align*}
\ds\int_{B_r(x_0)}|\nabla u|^2\,dx&\le \int_{B_r(x_0)}|\nabla (u+\varphi)|^2\,dx+\Lambda|B_r|+\beta\int_{\partial B_r(x_0)}u^2\,d\HH^{d-1}\\
&\le \int_{B_r(x_0)}|\nabla (u+\varphi)|^2\,dx+\Big(\Lambda\omega_d\delta +\beta d\omega_dM^2\Big)r^{d-1},
\end{align*}
for every $\varphi\in H^1_0(B_r(x_0))$ such that $\varphi\ge 0$ in $B_r(x_0)$. Applying \cref{l:holder1}, we obtain 
$$\|u\|_{L^\infty(B_r(x_0))}\le C_d\big(\Lambda +\beta M^2\big)^{\sfrac12}\,{r^{\sfrac{1}2}}\quad\text{for every}\quad r\le \frac{\delta}{8}.$$	
Using this estimate in \cref{l:outwards}, we get that, for every $r\in(0,\sfrac{\delta}{8})$,
\begin{align*}
\ds\int_{B_r(x_0)}|\nabla u|^2\,dx&\le \int_{B_r(x_0)}|\nabla (u+\varphi)|^2\,dx+\Lambda|B_r|+\beta\int_{\partial B_r(x_0)}u^2\,d\HH^{d-1}\\
&\le \int_{B_r(x_0)}|\nabla (u+\varphi)|^2\,dx+C_d\Big(\Lambda +\beta\big(\Lambda +\beta M^2\big)\Big)r^d,\end{align*}
for every nonnegative $\varphi\in H^1_0(B_r(x_0))$. Applying again \cref{l:holder1}, we get the claim.
\end{proof}

\section{Non-degeneracy of $u$}
In this section we show that $u$ is a subsolution of the Alt-Caffarelli functional. From this information, we can immediately deduce that $\{u>0\}$ has finite perimeter in $D_\delta$, for every $\delta>0$. Moreover, the suboptimality of $u$ implies that it is non-degenerate, which assures that the blow-up limits of $u$ are not identically zero.
\begin{lemma}\label{l:inwards}
	Let $u\in H^1(\R^d)$ be the function from \cref{sub:u_infty}. Then, 
	\begin{equation*}
	\ds\int_{D}|\nabla u|^2\,dx+\Lambda\big|\{u>0\}\cap D\big|\le \int_{D}|\nabla v|^2\,dx+\Lambda\big|\{v>0\}\cap D\big|,
	\end{equation*}
	for any $v\in H^1(\R^d)$ such that 
	$$v-u\in H^1_0(D)\qquad\text{and}\qquad 0\le v\le u\quad\text{in}\quad D.$$
\end{lemma}	
\begin{proof}
	Testing the optimality of $\big(u_{\eps_n},\Omega_{\eps_n}^1,\Omega_{\eps_n}^2\big)$ with $\big(v,\Omega_{\eps_n}^1,\Omega_{\eps_n}^2\big)$, we obtain
	\begin{equation*}
	\ds\int_{D}|\nabla u_{\eps_n}|^2\,dx+\Lambda|\{u_{\eps_n}>0\}\cap D|\le \int_{D}|\nabla  v|^2\,dx+\Lambda|\{v>0\}\cap D|.
	\end{equation*}
	Passing to the limit as $n\to\infty$, we get the claim.
\end{proof}	

As an immediate consequence, we have 

\begin{cor}\label{l:downgrowth}
	Let $u\in H^1(\R^d)$ be the function defined in \cref{sub:u_infty}. Then:
	\begin{enumerate}[\quad\rm(i)]
	\item the set $\{u>0\}$ has locally finite perimeter in $D$;
	\item there is a constant $\eta>0$ such that for every $x_0\in D\cap \overline{\{u>0\}}$
	$$\|u\|_{L^\infty(B_r(x_0))}\ge \eta\,r\qquad\text{for every}\qquad B_{2r}(x_0)\subset D.$$
	\end{enumerate}	 
\end{cor}
\begin{proof}
See \cite{altcaf} or \cite{velectures}.	
\end{proof}	

\section{Density estimate and its consequences}\label{s:density}
In this section, we will show that the free boundary $\partial\{u>0\}$ is not touching $\partial E_1$ and $\partial E_2$. This is a crucial step in proving that $\Omega_1$ and $\Omega_2$ have finite perimeter in $\R^d$. The main result is the following.
\begin{prop}[Non-collapsing]\label{p:non-collapsing}
	Let $u\in H^1(\R^d)$ be the function from \cref{sub:u_infty}. Then, there is a positive constant $t>0$ such that $u\ge t$ in a neighborhood of $\overline E_1\cup \overline E_2$.
\end{prop}	
The proof of \cref{p:non-collapsing} is based on the following lemma. 
\begin{lemma}[Density estimate]\label{l:density_estimate}
	Let $u\in H^1(\R^d)$ be the function from \cref{sub:u_infty}. There is a constant $c>0$ and $R_0>0$, depending only on $d$, $\beta$, $\Lambda$ and $\|g\|_{L^\infty}$, such that
	$$\big|B_R(x_0)\cap \{u=0\}\big|\ge c|B_R|\quad\text{for every}\quad B_R(x_0)\subset D\quad\text{with}\quad u(x_0)=0.$$
\end{lemma}	
\begin{proof}
We notice that by \cref{l:upgrowth}, we have that	
\begin{equation}\label{e:dens-lemma-bound-from-above}
\|u\|_{L^\infty(B_\rho(x_0))}\le C\rho\quad\text{for every}\qquad \rho<\frac{R}{256}.
\end{equation}
Now, we consider the competitor 
$$\widetilde \Omega_1:=\Omega_1\cup B_{\sfrac{R}{256}}(x_0)\ ,\quad \widetilde \Omega_2:=\Omega_2\setminus B_{\sfrac{R}{256}}(x_0)\ ,\quad \widetilde u(x):=\begin{cases}u(x)\quad \text{if}\quad x\in \R^d\setminus B_{\sfrac{R}{256}}(x_0)\\
h(x)\quad \text{if}\quad x\in B_{\sfrac{R}{256}}(x_0)
\end{cases},$$
$h$ being the harmonic extension of $u$ in $B_{\sfrac{R}{256}}(x_0)$.
Setting for simplicity $r=\frac{R}{256}$ and $x_0=0$, by \cref{l:outwards}, we have that 
	\begin{align*}
\ds\int_{B_r}|\nabla u|^2\,dx+\Lambda|B_r\cap\{u>0\}|
\le \int_{B_r}|\nabla h|^2\,dx+\Lambda|B_r\cap\{h>0\}|+\beta d\omega_dr^{d-1}C^2r^2.
\end{align*}
The rest of the proof follows the analogous lemma from \cite{altcaf}. Using the fact that $h$ is harmonic and strictly positive in $B_r$, we have 
	\begin{align*}
\ds\int_{B_r}|\nabla (u-h)|^2\,dx \le \Lambda|B_r\cap\{u=0\}|+\beta C^2 d\omega_dr^{d+1}.
\end{align*}
By the Poincaré inequality, there is a dimensional constant $C_d$ such that 
	\begin{align*}
\left(\frac1{|B_r|}\int_{B_r}(h-u)\,dx\right)^2\le \frac1{|B_r|}\int_{B_r}(h-u)^2\,dx \le \frac{C_dr^2}{|B_r|}\Big(\Lambda|B_r\cap\{u=0\}|+\beta C^2 r^{d+1}\Big).
\end{align*}
On the other hand, the combination of \cref{l:downgrowth} and the subharmonicity of $u$ in $B_R$, gives that for any $\kappa\in(0,1)$
$$\frac1{|B_{\kappa r}|}\int_{B_{\kappa r}}h(x)\,dx=\mean{\partial B_r}u\ge 2^{-d}cr.$$
On the other hand, \eqref{e:dens-lemma-bound-from-above} implies that
$$\frac1{|B_{\kappa r}|}\int_{B_{\kappa r}}u(x)\,dx\le C\kappa r.$$ 
Thus, 
	\begin{align*}
(c-C2^d\kappa)^2\le \frac{C_d}{|B_r|}\Big(\Lambda|B_r\cap\{u=0\}|+\beta C^2 r^{d+1}\Big).
\end{align*}
Choosing $\kappa$ and $R>0$ small enough, we get the claim.
 \end{proof}	

\begin{proof}[\bf Proof of \cref{p:non-collapsing}]
Suppose that $x_0\in D$ is such that 
$$u(x_0)=0\qquad\text{and}\qquad \delta:=\text{\rm dist}(x_0,\partial E_1)<\delta_0,$$
where the constant $\delta_0>0$ will be chosen later. Let $y_0$ be the projection of $x_0$ on $E_1$ and let $R:=2|x_0-y_0|$. We consider the competitor
$$\widetilde \Omega_1:=\Omega_1\cup B_{R}(y_0)\ ,\quad \widetilde \Omega_2:=\Omega_2\setminus B_{R}(y_0)\ ,\quad \widetilde u(x):=\begin{cases}u(x)\quad \text{if}\quad x\in \R^d\setminus \big(B_{R}(y_0)\cap D\big)\\
h(x)\quad \text{if}\quad x\in B_{R}(y_0)\cap D
\end{cases},$$
$h$ being the harmonic extension of $u$ in $B_{R}(y_0)\cap D$. We notice that $\widetilde u$ is the solution to 
$$\min\Big\{\int_{B_R(y_0)}|\nabla v|^2\,dx\ :\ u-v\in H^1_0\big(B_R(y_0)\big)\Big\}.$$
Thus, by \cite[Lemma 3.7]{velectures} and the fact that $u=g$ on $E_1$, we have that 
\begin{equation}\label{e:potential_estimate}
\frac{1}{R^2}\big|\{u=0\}\cap B_R(y_0)\big|\left(\mean{\partial B_R(y_0)}{u\,d\HH^{d-1}}\right)^2\leq C_d\int_{B_R(y_0)}|\nabla (u-h)|^2\,dx.
\end{equation}
On the other hand, using Lemma \ref{l:outwards}, we get 
\begin{equation}\label{e:non-collapsing-minimality}
\ds\int_{B_R(y_0)}|\nabla (u-h)|^2\,dx\le \Lambda|B_R(y_0)\cap\{u=0\}|+\beta\int_{\partial B_R(y_0)}u^2d\HH^{d-1}.
\end{equation}
 Now, since $u\equiv 1$ on $E_1$ and $u\le 1$ in $\R^d$, we get that 
$$\frac{1}{C_1}\le \mean{\partial B_R(y_0)}{u\,d\HH^{d-1}}\qquad\text{and}\qquad \int_{\partial B_R(y_0)}u^2d\HH^{d-1}\le d\omega_dR^{d-1},$$
where $C_1>0$ is a constant depending only on $E_1$ (notice that since $E_1$ is regular, for $\delta>0$ small, we can choose $C_1\simeq 2$). Thus, combining \eqref{e:potential_estimate} and \eqref{e:non-collapsing-minimality}, we get 
\begin{equation*}
\frac{1}{R^2}\big|\{u=0\}\cap B_R(y_0)\big|\le C_1^2 C_d\Big(\Lambda+\frac{\beta}{R}\Big)|B_R|\,,
\end{equation*}
which by the density estimate \cref{l:density_estimate}, gives a contradiction when $R$ is small enough.
\end{proof}	
\section{Regularity of the free boundary $\partial\{u>0\}$ }\label{s:regularity}

In this section we prove the following. 
\begin{teo}\label{t:regularity-free-boundary}
	Let $u\in H^1(\R^d)$ be the function defined in \cref{sub:u_infty} and let $d=2$. Then $\partial \{u>0\}\cap D$ is a $C^{1,\alpha}$ regular $(d-1)$-dimensional manifold.
\end{teo}	
The proof is based on the fact that $u$ satisfies an almost-minimality condition in $D$. Precisely, by \cref{l:holder1} and \cref{l:outwards}, we have that, given a compact set $K\subset D$, there are
constants $C>0$ and $r_0>0$ for which 
\begin{align}
\ds\int_{B_r(x_0)}|\nabla u|^2\,dx&+\Lambda|B_r(x_0)\cap\{u>0\}|\notag\\
&\le \int_{B_r(x_0)}|\nabla v|^2\,dx+\Lambda|B_r(x_0)\cap\{v>0\}|+Cr^{d+1},\label{e:almost-minimality}
\end{align}
for every $r<r_0$,  $x_0\in\partial\{u>0\}\cap D$ and $v\in H^1(B_r(x_0))$ such that $u-v\in H^1_0(B_r(x_0))$. In \cref{sub:blow-up}, we use the almost-minimality to show that when $d\le 4$ every blow-up $u$ is a half-plane solution (that is, solution of the form \eqref{e:blow-up-limits}) of the classical Alt-Caffarelli functional. Then, in \cref{sub:epi}, we show that in dimension $d=2$, we can use the epiperimetric inequality from \cite{SV} to conclude the proof of \cref{t:regularity-free-boundary}.
\begin{oss}
We notice that the function $u$ might not be smooth in the open set $\{u>0\}$. In fact, $u$ is not even $C^1$ as the gradient is not continuous across $\partial\Omega_1\cap\partial\Omega_2$. We stress that we can still use the 2D epiperimetric inequality from \cite{SV} together with the almost-minimality condition \eqref{e:almost-minimality} to prove the $C^{1,\alpha}$ regularity of the free boundary, but we cannot improve this regularity to $C^\infty$. 
\end{oss}

\begin{oss}
	We expect \cref{t:regularity-free-boundary} to hold in every dimension $2\le d\le 4$, as there are several epsilon-regularity results for functions $u$ satisfying almost-minimality conditions similar to \eqref{e:almost-minimality} (see for instance \cite{DT,DET,DS,STV}), but we stress that non of these results directly apply to \eqref{e:almost-minimality}. In fact, the almost-minimality of $u$ only holds around points at the boundary $\partial\{u>0\}$ (and in our case $u$ is not even $C^{1,\alpha}$ in $\{u>0\}$), which essentially requires \cite{DT,DET,DS,STV} to be revisited in order to be used in our context. We choose the approach from \cite{STV} which limits \cref{t:regularity-free-boundary} to the case $d=2$, but on the other hand is based on the epiperimetric inequality from \cite{SV}, which works without any modifications in our case.
\end{oss}

\subsection{Blow-up sequences and blow-up limits}\label{sub:blow-up}
Let $x_0\in\partial\{u>0\}\cap D$. We define
$$u_r(x):=\frac1ru(x_0+rx).$$
Let $r_n$ be an infinitesimal sequence. 
Then, for $n$ large enough, the sequence 
$u_{r_n}$ is uniformly bounded in $L^\infty$ in every ball $B_{2R}\subset\R^d$. Moreover, by \cref{l:holder2}, $u_{r_n}$ is uniformly bounded also in $C^{0,\sfrac13}(B_R)$. Thus, up to a (non-relabelled) subsequence $u_{r_n}$ it converges uniformly in $B_R$). By a diagonal sequence argument, there is a continuous function 
$$u_0:\R^d\to\R$$
and a subsequence $u_{r_n}$ such that $u_{r_n}$ converges to $u_0$ uniformly on every ball $B_R\subset \R^d$. We will say that $u_0$ is a blow-up limit of $u$ at $x_0$. 

\begin{prop}
Let $2\le d\le 4$ and let $u$ be the function from \cref{sub:u_infty}. Then every blow-up limit $u_0:\R^d\to\R$ of $u$ at a point $x_0\in\partial\{u>0\}\cap D$ is of the form 
\begin{equation}\label{e:blow-up-limits}
u_0(x)=\sqrt{\Lambda}\, (x\cdot\nu)_+\ ,
\end{equation}
for some unit vector $\nu\in\R^d$.
\end{prop}	
\begin{proof}
Let $u_{r_n}$ be a blow-up sequence converging to $u_0$. We notice that, by \cref{l:downgrowth}, $u_0$ is non-trivial. Moreover, using the almost-minimality condition \eqref{e:almost-minimality} we get that 
 $u_0$ is a local minimizer of the Alt-Caffarelli functional (see for instance \cite{altcaf} of \cite{velectures}). Precisely, 
$$\int_{B_R}|\nabla u|^2\,dx+\Lambda|B_R\cap\{u>0\}|\le \int_{B_R}|\nabla v|^2\,dx+\Lambda|B_R\cap\{v>0\}|,$$
for every $B_R\subset\R^d$ and every $v\in H^1(B_R)$ such that $u-v\in H^1_0(B_R)$. Moreover, the almost-minimality condition \eqref{e:almost-minimality} implies that every blow-up limit $u_0$ is $1$-homogeneous (see \cite{STV}). When $d\le 4$, using \cite{CJK} and \cite{JS}, this gives that every blow-up limit $u_0$ is of the form \eqref{e:blow-up-limits}.
\end{proof}

\subsection{Epiperimetric inequality and regularity of $\partial\{u>0\}$}\label{sub:epi}
For any $\varphi\in H^1(B_1)$ we consider the Weiss' boundary adjusted energy introduced in \cite{weiss}
$$W(\varphi):=\int_{B_1}|\nabla \varphi|^2\,dx+\Lambda|\{\varphi>0\}\cap B_1|-\int_{\partial B_1}\varphi^2\,d\HH^{d-1}.$$
Let $K$ be a compact set contained in $D$ and let $C>0$ and $r_0$ be the constants from the almost-minimality condition \eqref{e:almost-minimality}. Let $x_0\in\partial\{u>0\}\cap K$ be fixed and let $u_r(x):=\frac1ru(x_0+rx)$. Then, the derivative of $W(u_r)$ is given by
$$\frac{\partial}{\partial r}W(u_r)=\frac{d}{r}\big(W(z_r)-W(u_r)\big)+\frac{1}{r}\int_{\partial B_1}|x\cdot\nabla u-u|^2\,d\HH^{d-1}(x),$$
where $z_r:B_1\to\R$, $z_r(x):=|x|u_r\big(\sfrac{x}{|x|}\big)$, is the $1$-homogeneous extension of $u_r$ in $B_1$. Now, by the 2D epiperimetric inequality of \cite{SV}, we have that there is a constant $\eps\in(0,1)$ such that for every $r>0$, there exists a function $h_r:B_1\to\R$ with $h_r=u_r=z_r$ on $\partial B_1$ and 
$$W(h_r)-\Theta\le (1-\eps)\big(W(z_r)-\Theta\big)\,,\quad \text{where} \quad\Theta:=\Lambda\frac{|B_1|}2.$$
Now, using the almost minimality \eqref{e:almost-minimality} of $u$, we have that for every $r\le r_0$
\begin{align*}
\frac{\partial}{\partial r}W(u_r)&= \frac{d}{r}\big(W(z_r)-W(u_r)\big)+\frac1r\int_{\partial B_1}|x\cdot\nabla u-u|^2\,d\HH^{d-1}(x)\\
&= \frac{d}{r}\Big(\big(W(z_r)-\Theta\big)-\big(W(u_r)-\Theta\big)\Big)+\frac1r\int_{\partial B_1}|x\cdot\nabla u-u|^2\,d\HH^{d-1}(x)\\
&\ge  \frac{d}{r}\Big(\frac1{1-\eps}\big(W(h_r)-\Theta\big)-\big(W(u_r)-\Theta\big)\Big)+\frac1r\int_{\partial B_1}|x\cdot\nabla u-u|^2\,d\HH^{d-1}(x)\\
&\ge  \frac{d}{r}\Big(\frac1{1-\eps}\big(W(u_r)-\Theta\big)-\frac{Cr}{1-\eps}-\big(W(u_r)-\Theta\big)\Big)+\frac1r\int_{\partial B_1}|x\cdot\nabla u-u|^2\,d\HH^{d-1}(x)\\
&\ge  \frac{d}{r}\frac{\eps}{1-\eps}\big(W(u_r)-\Theta\big)-\frac{dC}{1-\eps}+\frac1r\int_{\partial B_1}|x\cdot\nabla u-u|^2\,d\HH^{d-1}(x).
\end{align*}
Taking $\ds\gamma=\frac{d\,\eps}{1-\eps}$, we get that 
$$\frac{\partial}{\partial r}\left(\frac{W(u_r)}{r^\gamma}+\frac{dCr^{1-\gamma}}{1-(d+1)\eps}\right)\ge \frac1{r^{1+\gamma}}\int_{\partial B_1}|x\cdot\nabla u-u|^2\,d\HH^{d-1}(x),$$
which implies that 
$$W(u_r)\le C_0r^\gamma\ ,\qquad\text{where}\qquad C_0:=\frac{W(u_{r_0})}{r_0^\gamma}+\frac{dCr_0^{1-\gamma}}{1-(d+1)\eps}.$$
By a standard argument (see for instance \cite{SV}, \cite{STV} or \cite{velectures}), this implies the uniqueness of the blow-up limit at $x_0$ and the $C^{1,\alpha}$-regularity of $\partial\{u>0\}$ in $K$, which concludes the proof of \cref{t:regularity-free-boundary}.\qed

\section{Regularity of the free interface}

In this section we prove the following theorem. 

\begin{teo}\label{t:regularity-free-interface}
Let $u$ be the function from \cref{sub:u_infty} and let $\Omega_1$ and $\Omega_2$ be the sets constructed in \cref{sub:omega12}. Then,
\begin{enumerate}[\quad\rm(i)]
\item in any dimension $d\ge 2$, the free interface $\partial^\ast\Omega_1\cap \big(D\cap\{u>0\}\big)$ is a $C^{\infty}$ manifold up to a closed singular set of Hausdorff dimension at most $d-8$;
\item in dimension $d=2$, the contact set $\partial\Omega_1\cap\partial\Omega_2\cap (D\cap\partial\{u>0\})$ is discrete in $D$;
\item in dimension $d=2$, in a neighborhood of every point $x\in \partial\Omega_1\cap\partial\Omega_2\cap (D\cap\partial\{u>0\})$ the boundary  $\partial\Omega_1\cap\partial\Omega_2\cap\{u>0\}$ is a $C^\infty$ curve, $C^1$ regular up to the endpoint $x$, and attaches orthogonally to $\partial\{u>0\}$ at $x$.
\end{enumerate}
\end{teo}	

\subsection{Minimality and regularity of the free interface in $D\cap\{u>0\}$}

\begin{prop}[Minimality of the limit sets]\label{p:minimality-omega12}
In any $d\ge 2$, let $\Omega_1$ and $\Omega_2$ be the sets from \cref{sub:omega12} and $u$ be the function from \cref{sub:u_infty}. Then, for every open set $A\subset\subset D\cap\{u>0\}$, 
	\begin{equation}\label{e:minimality-omega12}
	\int_{A\cap\partial^\ast\Omega_1}u^2\,d\HH^{d-1}\le \int_{A\cap\partial^\ast\widetilde \Omega_1}u^2\,d\HH^{d-1}\,,
	\end{equation}
	for every set $\widetilde\Omega_1$ of locally finite perimeter in $D\cap\{u>0\}$ such that $\widetilde\Omega_1\Delta\Omega_1\subset\subset A$.
	
	 In particular, $\partial^\ast\Omega_1\cap \big(D\cap\{u>0\}\big)$ is a $C^{1,\alpha}$-regular manifold up to a closed singular set of Hausdorff dimension at most $d-8$.
\end{prop}
\begin{proof}
Without loss of generality we can suppose that $A$ is a finite union of balls 
$$A=\bigcup_{j=1}^NB_{r_j}(x_j).$$
Let $(u_n,\Omega_n^1,\Omega_n^2):=(u_{\eps_n},\Omega_{\eps_n}^1,\Omega_{\eps_n}^2)$ be the sequence of minimizers from \cref{sub:u_infty} and \cref{sub:omega12} converging to $(u,\Omega_1,\Omega_2)$. Then, by construction 
$$\Omega_n^1\cap\Omega_n^2=\emptyset\qquad\text{and}\qquad \Big(\Omega_n^1\cap A\Big)\cup\Big(\Omega_n^2\cap A\Big)=A,$$
and the same holds for the limit sets $\Omega_1$ and $\Omega_2$. Let now $\widetilde\Omega_1$ be such that $\widetilde\Omega_1\Delta\Omega_1\subset\subset A$. Notice that we can find a family of balls $B_{\rho_j}(x_j)$, $j=1,\dots,N$, such that 
$$B_{\rho_j}(x_j)\subset B_{r_j}(x_j)\qquad\text{for every}\qquad j=1,\dots,N\,;$$
$$\widetilde\Omega_1\Delta\Omega_1\subset\subset \bigcup_{j=1}^NB_{\rho_j}(x_j)\subset\subset \bigcup_{j=1}^{N}B_{r_j}(x_j)\,,$$
and such that for every $j=1,\dots,N$, we have 
$$\HH^{d-1}\Big(\partial^\ast\Omega_1\cap \partial B_{\rho_j}(x_j)\Big)=0\qquad\text{and}\qquad \HH^{d-1}\Big(\partial^\ast\Omega_n^1\cap \partial B_{\rho_j}(x_j)\Big)=0\quad\text{for every}\quad n\ge 1.$$
We define 
$$B:=\bigcup_{j=1}^NB_{\rho_j}(x_j),$$
and consider the sets 
$$\widetilde\Omega_n^1:=\Big(\widetilde\Omega_1\cap B\Big)\cup\Big(\Omega_n^1\setminus B\Big)\qquad\text{and}\qquad \widetilde\Omega_n^2:=\Big(B\setminus\widetilde\Omega_1\Big)\cup\Big(\Omega_n^2\setminus B\Big).$$
Testing the optimality of $(u_n,\Omega_n^1,\Omega_n^2)$ with $(u_n,\widetilde\Omega_n^1,\widetilde\Omega_n^2)$, we obtain 
$$\int_{B\cap\partial^\ast\widetilde\Omega_1}u_n^2+\int_{\partial^\ast\Omega_n^1\setminus B}u_k^2=\int_{\partial^\ast\widetilde\Omega_n^1}u_n^2\ge \int_{\partial^\ast\Omega_n^1}u_n^2=\int_{B\cap\partial^\ast\Omega_n^1}u_n^2+\int_{\partial^\ast\Omega_n^1\setminus B}u_n^2\,$$
which we can write simply as
$$\int_{B\cap\partial^\ast\widetilde\Omega_1}u_n^2\ge \int_{B\cap\partial^\ast\Omega_n^1}u_n^2\,.$$ 
Using \cref{l:semicontinuity} and the uniform convergence of $u_n$ to $u$, we get 
$$\int_{B\cap\partial^\ast\widetilde\Omega_1}u^2=\lim_{n\to\infty}\int_{B\cap\partial^\ast\widetilde\Omega_1}u_n^2\qquad\text{and}\qquad \int_{B\cap\partial^\ast\Omega_1}u^2\le  \liminf_{n\to\infty}\int_{B\cap\partial^\ast\Omega_n^1}u_n^2\,,$$
which concludes the proof of \eqref{e:minimality-omega12}. The regularity of $\partial\Omega_1$ in $A$ then follows (as in \cite{GLMV}) from the fact that $u\in C^{0,\alpha}(A)$ and $u\ge t>0$ in $A$, for some constant $t>0$.
\end{proof}	

\subsection{Minimality of the free interface up to the boundary of $\{u>0\}$} 

\begin{prop}[Minimality of the limit sets]\label{p:boundary-minimality-omega12}
In any $d\ge 2$, let $\Omega_1$ and $\Omega_2$ be the sets from \cref{sub:omega12} and $u$ be the function from \cref{sub:u_infty}. Then, for every ball $B_r(x_0)\subset\subset D$, 
	\begin{equation}\label{e:boundary-minimality-omega12}
	\int_{B_r(x_0)\cap\partial^\ast\Omega_1}u^2\,d\HH^{d-1}\le \int_{B_r(x_0)\cap\partial^\ast\widetilde \Omega_1}u^2\,d\HH^{d-1}\,,
	\end{equation}
	for every set $\widetilde\Omega_1$ such that $\widetilde\Omega_1\Delta\Omega_1\subset\subset B_r(x_0)$ and which is of finite perimeter in $B_r(x_0)$. 
\end{prop}
\begin{proof}	
Let $\widetilde\Omega_1$ be such that $\widetilde\Omega_1\Delta\Omega_1\subset\subset B_r(x_0)$. 
Since $u$ satisfies the minimality condition from \cref{l:inwards}, we have that there are a sequence $\delta_n\to0$ and a constant $C>0$ such that
$$\HH^{d-1}\Big(B_r(x_0)\cap \partial^\ast \{u>\delta_n\}\Big)\le C\qquad\text{for every}\qquad n\ge 1\,.$$
Moreover, we can suppose that
$$\HH^{d-1}\Big(\partial^\ast\{u>\delta_n\}\cap \partial^\ast\Omega_1\Big)=\HH^{d-1}\Big(\partial^\ast\{u>\delta_n\}\cap \partial^\ast\widetilde \Omega_1\Big)=0.$$
Now, consider the sets
$$A_n:=B_r(x_0)\cap\{u>\delta_n\}\qquad\text{and}\qquad \Omega_n:=\Big(A_n\cap\widetilde\Omega_1\Big)\cup \Big(\Omega_1\setminus A_n\Big).$$
Since $\Omega_n\Delta\Omega_1\subset\subset B_r(x_0)\cap\{u>0\}$, by \cref{p:minimality-omega12}, we have that 
$$\int_{B_r(x_0)\cap \partial^\ast\Omega_1}u^2\le \int_{B_r(x_0)\cap \partial^\ast\Omega_n}u^2,$$
which we write as 
\begin{align*}
\int_{\big(B_r(x_0)\cap \partial^\ast\Omega_1\big)\cap A_n}&u^2+\int_{\big(B_r(x_0)\cap \partial^\ast\Omega_1\big)\setminus A_n}u^2\\
&\le \int_{\big(B_r(x_0)\cap \partial^\ast\widetilde\Omega_1\big)\cap A_n}u^2+\int_{\big(B_r(x_0)\cap \partial^\ast\Omega_1\big)\setminus A_n}u^2+2\int_{B_r(x_0)\cap \partial^\ast A_n}u^2\\
&\le \int_{\big(B_r(x_0)\cap \partial^\ast\widetilde\Omega_1\big)\cap A_n}u^2+\int_{\big(B_r(x_0)\cap \partial^\ast\Omega_1\big)\setminus A_n}u^2+2\delta_n^2C,
\end{align*}
which gives 
\begin{align*}
\int_{\big(B_r(x_0)\cap \partial^\ast\Omega_1\big)\cap A_n}u^2\le \int_{\big(B_r(x_0)\cap \partial^\ast\widetilde\Omega_1\big)\cap A_n}u^2+2\delta_n^2C.
\end{align*}
Passing to the limit as $n\to\infty$, we get \eqref{e:boundary-minimality-omega12}. 
\end{proof}

\subsection{Regularity of the free interface up to the boundary $\partial\{u>0\}$} In this section we will need the $C^{1,\alpha}$ regularity of the free boundary $\partial\{u>0\}$ in $D$, so in order to have \cref{t:regularity-free-boundary} we assume that $d=2$. 

Let $x_0\in D\cap\{u>0\}$ and $B_r(x_0)$ be a (small) ball contained in $D$. Without loss of generality, we suppose that $x_0=(0,0)$. We define the function $h:B_r\to\R$ as 
$$\Delta h=0\,\text{ in }\, B_r\cap\{u>0\}\,,\quad h=u\,\text{ in }\,  \partial B_r\,,\quad  h=0\,\text{ in }\,  B_r\cap\partial\{u>0\}\,.$$
Then, $h$ is $C^{1,\alpha}$ regular in $B_r\cap\{u>0\}$ up to the boundary $B_r\cap\partial\{u>0\}$ and moreover, there is a $C^{0,\alpha}$-regular strictly positive function  
$$a:B_r\cap\overline{\{u>0\}}\to\R$$
such that 
$$a(x)=\frac{h(x)}{u(x)}\quad\text{for}\quad x\in B_r\cap\{u>0\}\,;\quad a(x)=\frac{|\nabla h|(x)}{\sqrt{\Lambda}}\quad\text{for}\quad x\in B_r\cap\partial\{u>0\}.$$
Moreover, choosing $r>0$ small enough the set $B_r\cap\{u>0\}$ is simply connected, so we can find a function 
$w:B_r\cap\{u>0\}\to\R$ such that 
$$\partial_xw=-\partial_yh\quad\text{and}\quad \partial_yw=\partial_xh\quad\text{ in }\quad B_r\cap\{u>0\}\,,$$
$w$ being defined as 
$$w(x,y)=\int_{\gamma_{x,y}}\Big(\partial_yh\,dx-\partial_xh\,dy\Big),$$
where $\gamma_{x,y}:[0,1]\to\R$ is any $C^1$ curve connecting $(0,0)$ to $(x,y)$ in $B_r\cap\{u>0\}$. Then, the map 
$$\Phi:B_r\cap\overline{\{u>0\}}\ ,\qquad \Phi(x,y):=\Big(w(x,y),h(x,y)\Big),$$
is $C^{1,\alpha}$ smooth in $B_r\cap\overline{\{u>0\}}$ (up to the boundary $B_r\cap\partial{\{u>0\}}$) and the set 
$$A:=\Phi\Big(B_r\cap\overline{\{u>0\}}\Big)$$
is a relatively open subset of the upper half-plane $\{(w,h)\in\R^2\ :\ h\ge 0\}$. We notice that for $r$ small enough the function $\Phi$ is invertible. Then, we define 
$$\varphi:A\cap\{(w,h)\in\R^2\ :\ h\ge 0\}\to\R,\qquad \varphi:=\frac1{|\nabla h|a^2}\circ\Phi^{-1}.$$
Then, $\varphi$ is $C^{0,\alpha}$ and is bounded from below by a positive constant. We will show that in the new coordinates the set $\Omega:=\Phi(\Omega_1)$ locally minimizes the functional 
$$\mathcal F_2(\Omega):=\int_{\{h\ge 0\}\cap\partial^\ast \Omega}h^2\varphi(w,h)\,d\HH^1(w,h).$$
In fact, since $\partial^\ast\Omega_1$ is a $C^1$ curve, it is sufficient to check that for any 
$$\gamma:[0,1]\to B_r\cap\overline{\{u>0\}},\quad\gamma(t)=\big(x(t),y(t)\big),$$
we have 
\begin{align*}
\int_{\Phi(\gamma)}h^2\varphi(w,h)&:=\int_0^1h^2\big(\gamma(t)\big)\varphi\big(\Phi(\gamma(t))\big)\sqrt{\big(x'\partial_xw+y'\partial_yw\big)^2+\big(x'\partial_xh+y'\partial_yh\big)^2}\,dt\\
&=\int_0^1h^2\big(\gamma(t)\big)\varphi\big(\Phi(\gamma(t))\big)\sqrt{\big(x'\partial_yh-y'\partial_xh\big)^2+\big(x'\partial_xh+y'\partial_yh\big)^2}\,dt\\
&=\int_0^1h^2\big(\gamma(t)\big)\varphi\big(\Phi(\gamma(t))\big)\sqrt{(\partial_xh)^2+(\partial_yh)^2}\sqrt{(x'(t))^2+(y'(t))^2}\,dt\\
&=\int_0^1u^2\big(\gamma(t)\big)\sqrt{(x'(t))^2+(y'(t))^2}\,dt,
\end{align*}
which concludes the proof. In order to conclude the proof of the $C^{1}$ regularity of $\partial\Omega_1$, it is sufficient to prove that at the point $(w_0,h_0)=(0,0)$ the set 
$\Omega:=\Phi(\Omega_1)$ has a unique blow-up limit given by 
$$\Omega_0:=\{(w,h)\in\R^2\ :\ h\ge 0,\ w>0\}.$$
In order to do so, we consider the set 
$$\mathcal R:=\Big\{(w,X)\in\R\times\R^3\ :\ (w,|X|)\in \Omega\Big\}.$$
Then, $\mathcal R$ is a local minimizer of the functional 
$$\mathcal F_4(\mathcal R):=\int_{\partial^\ast R}\varphi(w,|X|)\,d\HH^3(w,X),$$
among all sets with the same simmetries as $\mathcal R$, that is, all sets of the form 
$$\widetilde{\mathcal R}:=\Big\{(w,X)\in\R\times\R^3\ :\ (w,|X|)\in \widetilde\Omega\Big\},$$
for some $\widetilde\Omega\subset\{(w,h)\in\R^2\ :\ h\ge 0\}$. Now, by the monotonicity formula for the local minimizers of the area (see for instance \cite{maggi}), we have that any blow-up limit $\mathcal R_0$ of $\mathcal R$ is a cone in $\R^4$, which is area-minimizing with respect to perturbations that preserve the simmetries of $\mathcal R$. But then, since the dimension of $\partial\mathcal R_0$ is less than $7$, we have that $\partial\mathcal R_0$ is necessary a plane (with the same simmetries as $\mathcal R$). Thus, $\partial\mathcal R_0$ is ortogonal to the line $\{0,0,0\}\times\R$, which concludes the proof of the uniqueness of the blow-up, which implies points (ii) and (iii) of \cref{t:regularity-free-interface}.

\section{Proof of Theorem \ref{t:main}}\label{s:proof12}

In order to prove the existence of a solution to \eqref{problema}, we observe that as a consequence of an almost-minimality condition involving the one-phase Alt-Caffarelli functional and \cref{t:regularity-free-interface}, we have that the sets $\Omega_1$ and $\Omega_2$ constructed in \cref{sub:omega12} have locally finite perimeter in $D$.
It remains to prove that  
	\begin{enumerate}[\quad\rm(i)]
	\item $\Omega_1$ and $\Omega_2$ are sets of finite perimeter in $\R^d$;
	\item $(u,\Omega_1,\Omega_2)$ is a solution to \eqref{problema}.
	\end{enumerate}
First, for $j=1,2$ and $\delta>0$, we define 
$$E_{j}^\delta:=\Big\{x\in \R^d\setminus E_j\ :\ \text{\rm dist}(x,E_j)>\delta\Big\}.$$
Then, $E_1^\delta\cap\Omega_1$ and $E_2^\delta\cap\Omega_2$ are sets of finite perimeter and, by \cref{p:non-collapsing}, 
$$\text{\rm Per}(E_1^\delta\cap\Omega_1)+\text{\rm Per}(E_2^\delta\cap\Omega_2)<C,$$
where $C$ is a constant that does noit depend on $\delta$. Thus, passing to the limit as $\delta\to0$, we get that $\Omega_1$ and $\Omega_2$ have finite perimeter. Next, in order to prove (ii), we consider $\widetilde u\in\mathcal V$ and $(\widetilde\Omega_1,\widetilde\Omega_2)\in\mathcal A(u)$. Testing the optimality of $(u_\eps,\Omega_\eps^1,\Omega_\eps^2)$ we get that 
$$J_\eps(u_\eps,\Omega_\eps^1,\Omega_\eps^2)\le J_\eps (\widetilde u,\widetilde\Omega_1,\widetilde\Omega_2).$$
Now, the semicontinuity lemma (\cref{l:semicontinuity}) gives that 
 $$J_{\beta,\Lambda}(u,\Omega_1,\Omega_2)=\liminf_{\eps\to0}J_\eps(u_\eps,\Omega_\eps^1,\Omega_\eps^2)\le \lim_{\eps\to0} J_\eps (\widetilde u,\widetilde\Omega_1,\widetilde\Omega_2)=J_{\beta,\Lambda}(\widetilde u,\widetilde\Omega_1,\widetilde\Omega_2),$$
 which concludes the proof of the existence.	
 
Moreover, if $(u,\Omega_1,\Omega_2)$ is any solution to \eqref{problema}, then it satisfies the minimality conditions from \cref{l:outwards}, \cref{l:inwards}, \cref{p:minimality-omega12} and \cref{p:boundary-minimality-omega12}, so by \cref{t:regularity-free-boundary} and \cref{t:regularity-free-interface}, the claims (i) and (ii) of \cref{t:main} follow.

\section*{Acknowledgments}
\noindent
The first author was partially supported by  PRIN 2017 {\it Nonlinear
	Differential Problems via Variational, Topological and Set-valued Methods} (Grant 2017AYM8XW).
The third author has been supported by the European Research Council (ERC) under the European Union's Horizon 2020 research and innovation programme (grant agreement VAREG, No. 853404).

\end{document}